\documentclass[11pt]{amsproc}
\usepackage{mathrsfs}
\usepackage{stmaryrd}
\usepackage{cases}
\usepackage{amsfonts}
\usepackage{graphicx}
\usepackage{amsmath,amstext,amsbsy,amssymb, color}

\newtheorem{theorem}{Theorem}[section]
\newtheorem{lemma}[theorem]{Lemma}
\newtheorem{proposition}[theorem]{Proposition}
\newtheorem{corollary}[theorem]{Corollary}
\theoremstyle{definition}

\theoremstyle{remark}

\numberwithin{equation}{section} \errorcontextlines=0

\newcommand{\diag}{\mbox{diag}}

\newcommand{\ot}{\otimes}

\newcommand{\GL}{\mathrm{GL}}
\newcommand{\gl}{\mathfrak{gl}}
\newcommand{\g}{\mathfrak{g}}

\newcommand{\sdet}{\mathrm{sdet}}
\newcommand{\wt}{\widetilde}
\newcommand{\wh}{\widehat}
\newcommand{\qin}{q^{-1}}
\newcommand{\ol}{\overline}
\newcommand{\tw}{\mathrm{tw}}
\begin{document}

\title[Twisted q-Yangians and Sklyanin determinants]
{Twisted q-Yangians and Sklyanin determinants}
\author{Naihuan Jing}
\address{Department of Mathematics, North Carolina State University, Raleigh, NC 27695, USA}
\email{jing@ncsu.edu}

\author{Jian Zhang}
\address{School of Mathematics and Statistics, and Hubei Key Lab-Math. Sci., Central China Normal University, Wuhan, Hubei 430079, China}
\email{jzhang@ccnu.edu.cn}
\thanks{{\scriptsize
\hskip -0.6 true cm MSC (2020): 17B37; Secondary: 17B10, 81R50, 14M17, 15A15, 13A50
}}

\thanks{{\scriptsize\hskip -0.6 true cm *Corresponding author: Jian Zhang}}

\begin{abstract} $q$-Yangians can be viewed both as quantum deformations of the loop algebras of upper triangular Lie algebras and deformations of the
Yangian algebras. In this paper, we study the quantum affine algebra as a product of two copies of the $q$-Yangian algebras. This viewpoint
enables us to investigate the invariant theory of quantum affine algebras and their twisted versions. We introduce the twisted
Sklyanin determinant for twisted quantum affine algebras and establish various identities for the Sklyanin determinants.
\end{abstract}
\maketitle
\section{Introduction}

Let $\mathfrak g$ be a complex simple Lie algebra, the Yangian $\mathrm Y(\mathfrak g)$ was defined by Drinfeld \cite{Dr} as the
algebraic structure to solve the rational Yang-Baxter equation.
As an algebra, $\mathrm Y(\mathfrak{gl}_N)$ deforms the loop algebra $\mathfrak {gl}_N[t]$
and there are two important homomorphisms:
\begin{equation}\label{e:ev}
\mathrm U(\mathfrak{gl}_N)\hookrightarrow \mathrm  Y(\mathfrak{gl}_N), \qquad \mathrm Y(\mathfrak{gl}_N)\longrightarrow \mathrm U(\mathfrak gl_N).
\end{equation}
Thus many representation problems of $\mathrm U(\mathfrak{gl}_N)$ can be better understood over the Yangian algebra $\mathrm Y(\mathfrak{gl}_N)$,
see \cite{M} for examples.

However for other classical Lie algebras $\mathfrak g_N=\mathfrak o_{2n+1}, \mathfrak {sp}_{2n}, \mathfrak o_{2n}$,
there are no such evaluation homomorphisms $\mathrm Y(\mathfrak{g}_N)\longrightarrow \mathrm U(\mathfrak g_N)$. Olshanski introduced the twisted Yangian $\mathrm Y^{\mathrm {tw}}(\mathfrak g_N)$ for $\mathfrak g_N=\mathfrak o_{N}, \mathfrak {sp}_{N}$ as coideal subalgebras of $\mathrm Y(\mathfrak{gl}_N)$ corresponding to the orthogonal and symplectic types.
 It is known that $\mathrm Y^{\mathrm {tw}}(\mathfrak g_N)$ are
also deformations of the enveloping algebras $\mathrm U(\mathfrak g_N)$ and more importantly there are canonical homomorphisms
$$\mathrm U(\mathfrak{g}_N)\hookrightarrow \mathrm Y^{\mathrm {tw}}(\mathfrak{g}_N), \qquad \mathrm Y^{\mathrm {tw}}(\mathfrak{g}_N)\longrightarrow \mathrm U(\mathfrak g_N)$$
which provide the natural lifting and have important applications for studying $\mathrm U(\mathfrak g_N)$.

On the other hand, the enveloping algebra $\mathrm U(\mathfrak g)$ has another quantum deformation $\mathrm U_q(\mathfrak g)$ corresponding to the trigonometric Yang-Baxter $R$-matrix introduced by Drinfeld and Jimbo \cite{Dr, J}.

For a classical symmetric
pair $(\mathfrak g, \mathfrak{k})$, the twisted  quantum enveloping algebra
$\mathrm U^{\mathrm{tw}}_q(\mathfrak k)$ was introduced by Noumi \cite{No} and Dijkhuizen-Noumi-Sugitani \cite{DN, DNS}, where the algebra $\mathrm U^{\mathrm{tw}}_q(\mathfrak{so}_N)$ was studied earlier by \cite{GK}.
In the first part of the paper, we consider the twisted quantum enveloping
algebras $\mathrm U^{\mathrm{tw}}_q(\mathfrak {o}_N)$ and $\mathrm U^{\mathrm{tw}}_q(\mathfrak{sp}_N)$ corresponding to the symmetric pairs.
\begin{equation}
\begin{aligned}
\mathrm{AI}: \qquad & (\mathfrak{gl}_N, \mathfrak{o}_N),\\
\mathrm{AII}: \qquad & (\mathfrak{gl}_{2n}, \mathfrak{sp}_{2n}).
\end{aligned}
\end{equation}

In 2003, Molev, Ragoucy, and Sorba used the R-matrix method to formulate the $q$-Yangian $\mathrm Y_q(\mathfrak{gl}_N)$  as a $q$-deformation of
 $\mathrm Y(\mathfrak{gl}_N)$
by replacing the rational $R$-matrix by the spectral trigonometric $R$-matrix \cite{MRS}.
They introduced twisted q-Yangians $\mathrm Y^{\mathrm {tw}}_q(\mathfrak{g}_N)$ as coideal subalgebras of the quantum affine algebra $\mathrm U_q(\widehat{\mathfrak{gl}}_N)$ \cite{DF} corresponding to the
symmetric spaces of orthogonal and symplectic types. Note that unlike the Yangian case, the algebras
$\mathrm Y^{\mathrm {tw}}_q(\mathfrak{g}_N)$ are generated by
quadratic elements from both the upper and lower triangular subalgebras $\mathrm U^{\pm}_q(\widehat{\mathfrak{gl}}_N)$
according to the imaginary triangular decompositions
\begin{equation}
\widehat{\mathfrak g}=\widehat{\mathfrak g}^+\oplus \widehat{\mathfrak g}^0\oplus \widehat{\mathfrak g}^-
\end{equation}
where $\widehat{\mathfrak g}^{\pm}=\mathfrak g\otimes \mathbb C[t^{\pm 1}]t^{\pm 1}$. Later
Lu established the isomorphism between the twisted $q$-Yangians and affine $\imath $quantum groups associated to type   $\mathrm{AI}$ \cite{Lu}.
It turns out that there are natural homomorphisms:
\begin{align*}
&\mathrm U^{\mathrm {tw}}_q(\mathfrak{g}_N)\hookrightarrow \mathrm Y^{\mathrm {tw}}_q(\mathfrak{g}_N),\\
&\mathrm Y^{\mathrm {tw}}_q(\mathfrak{g}_N)\longrightarrow \mathrm U^{\mathrm {tw}}_q(\mathfrak g_N).
\end{align*}

A different approach to quantum symmetric spaces was described for arbitrary symmetric pair $(\mathfrak g, \mathfrak k)$ by Letzter \cite{L}, which
shows the importance of the coideal construction. In type $A$ case, the quantum minors and quantum determinants satisfy nice algebraic equations
that generalize the classical identities \cite{HH, ER, KL, bL}.
Recently we have studied the dual quantum symmetric spaces of finite types $\mathrm{AI}$ and $\mathrm{AII}$
for the quantum coordinate algebra
and determined their centers using the Sklyanin determinants and quantum Pfaffians \cite{JZ2}, and we have showed that
the Sklyanin determinants satisfy similar identities like the quantum determinant. It is natural to seek for
a similar construction for the quantum symmetric spaces associated with quantum affine algebras.

The aim of this paper is to study the quantum affine symmetric spaces of orthogonal and symplectic types
and also partly generalize the invariant theory of the Yangians \cite{M2} to the quantum affine algebras. After reviewing the quantum symmetric spaces
of type $\mathrm{AI}, \mathrm{AII}$ as in \cite{MRS},
we formulate the quantum affine symmetric spaces using the R-matrix method \cite{FRT, RTF} subject to certain reflective RTT equation and
introduce the notion of Sklyanin determinants in both cases.

In the third part of the paper we obtain several identities for the quantum Sklyanin determinant. The key identities can be expressed as
minor identities for the quantum Sklyanin determinant, which correspond to the classical identities \cite{BS} for the quantum determinant over the quantum general linear group.
Similar to the Yangian and twisted Yangians (cf. \cite{M2}), we will generalize the $q$-Jacobi identities, $q$-Cayley's complementary identities, the $q$-Sylvester identities and Muir's theorem to Sklyanin determinants
both in the quantum orthogonal and quantum symplectic situations.
In a sense, we have generalized several key identities for the general linear, orthogonal and symplectic groups to their counterparts to the
quantum affine symmetric spaces.

\section{Coideal subalgebras of $\mathrm{U}_q(\gl_N)$}
In this section we briefly recall the orthogonal and symplectic coideals $\mathrm{U}_{q}^{\tw} (\mathfrak{o}_{N})$ and $\mathrm{U}_{q}^{\tw} (\mathfrak{sp}_{N})$
of the
quantum algebra $\mathrm{U}_{q} (\mathfrak{gl}_{N})$  \cite{MRS} to prepare for our further study of the
quantum affine algebras.

Let $R$ be the matrix in $\mathrm{End}(\mathbb C^N\otimes \mathbb C^N)\simeq (\mathrm{End}\mathbb C^N)^{\ot 2}$ defined by:
\begin{align}\nonumber
&R=q\sum_{i}E_{ii}\otimes E_{ii}+ \sum_{i\neq j}E_{ii}\otimes E_{jj}
+(q-q^{-1})\sum_{i<j}E_{ij}\otimes E_{ji},
\end{align}
where $E_{ij}$ are the unit matrices in $\mathrm{End}\mathbb C^N$, and $q$ is a complex number often assumed to be not a root of unity. It is known that
$R$ satisfies the well-known Yang-Baxter equation:
\begin{align*}
R_{12}R_{13}R_{23}=R_{23}R_{13}R_{12},
\end{align*}
where $R_{ij}\in \mathrm{End}(\mathbb C^N\otimes \mathbb C^N\otimes \mathbb C^N)$ acts on the $i$th and $j$th copies of $\mathbb C^N$ as $R$ does on $\mathbb C^N\otimes \mathbb C^N$.

The {\it quantum enveloping  algebra} $\mathrm{U}_{q}(\gl_{N})$ is generated by the elements
$l_{ij}^{\pm}$
 with $1\leq i, j\leq N$ subject to the relations
\begin{align}
 l^{-}_{ij}&=l^+_{ji}=0 ,\ 1\leq i<j\leq N,\\
 l^+_{ii}l^-_{ii}&=l^-_{ii}l^+_{ii}=1,\ 1\leq i\leq N,\\ \label{RTT relation}
 RL_1^{\pm}L_2^{\pm}&=L_2^{\pm}L_1^{\pm}R,\ \quad  R L^+_1  L^-_2=L^-_2 L^+_1 R.\
\end{align}
Here $L^{\pm}$ are the matrices
 \begin{equation}
  L^{\pm}=\sum_{i,j}l^{\pm}_{ij}\otimes E_{ij},
 \end{equation}
which are regarded as elements in the algebra $\mathrm{U}_{q}(\gl_{N})\otimes \mathrm{End}(\mathbb{C}^{N})$. The elements in the $R$-matrix relation \eqref{RTT relation}
are elements in
$\mathrm{U}_{q}(\gl_{N})\otimes \mathrm{End}(\mathbb{C}^{N}) \otimes \mathrm{End}(\mathbb{C}^{N})$
and the subindices of $L^{\pm}$ indicate the copies of End $\mathbb{C}^{N}$ where $L^{\pm}$ act; e.g. $L^{\pm}_{1}=L^{\pm}\otimes 1$.
In terms of the generators the defining relations between the $l^{\pm}_{ij}$ can be written as

\begin{equation} \label{relation t}
q^{\delta_{ij}}l^{\pm}_{ia}l^{\pm}_{jb}-q^{\delta_{ab}}l^{\pm}_{jb}l^{\pm}_{ia}=(q-q^{-1})(\delta_{b<a}-\delta_{i<j})l^{\pm}_{ja}l^{\pm}_{ib}
\end{equation}
where $\delta_{i<j}$ equals 1 if $i<j$ and $0$ otherwise.

The relations involving both $l^+_{ij}$ and $l^-_{ij}$ have the following form
\begin{equation}
q^{\delta_{ij}}l^+_{ia}l^-_{jb}-q^{\delta_{ab}}l^-_{jb}l^+_{ia}
=(q-q^{-1})(\delta_{b<a}l^-_{ja}l^+_{ib}-\delta_{i<j}l^+_{ja}l^-_{ib}).
\end{equation}

\subsection{Orthogonal case}
The {\it twisted quantum enveloping algebra}
$\mathrm{U}_{q}^{\tw}(\mathfrak{o}_{N})$
is an unital algebra
generated by $s_{ij},1\leq i<j\leq N$ subject to the reflection relations
 \begin{equation}\label{q-orth}
 \begin{split}
 s_{ij}=0,\ 1\leq i<j\leq N,\\
 s_{ii}=1,\ 1\leq i\leq N, \\
 RS_1R^{t}S_2=S_2 R^{t}S_1R
 \end{split}
\end{equation}
where $S=(s_{ij})_{N\times N}$  and $R^t=R^{t_1}$ denotes the partial transpose in the first tensor factor:
\begin{equation}
R^{t_1}=q\sum_{1\leq i\leq N}e_{ii}\otimes e_{ii}+ \sum_{i\neq j}e_{ii}\otimes e_{jj}
+(q-q^{-1})\sum_{i>j}e_{ji}\otimes e_{ji}.
\end{equation}
The reflection relations can be written as
\begin{equation}
\begin{split}
q^{\delta_{aj}+\delta_{ij}}s_{ia}s_{jb}&-q^{\delta_{ab}+\delta_{ib}}s_{jb}s_{ia}
=(q-q^{-1})q^{\delta_{ai}}(\delta_{b<a}-\delta_{i<j})s_{ja}s_{ib}\\
&+(q-q^{-1})(q^{\delta_{ab}}\delta_{b<i}s_{ji}s_{ba}-q^{\delta_{ij}}\delta_{a<j}s_{ij}s_{ab})\\
&+(q-q^{-1})^{2}(\delta_{b<a<i}-\delta_{a<i<j})s_{ji}s_{ab}
\end{split}
\end{equation}
where $\delta_{i<j}$ or $\delta_{i<j<k}$ equals $1$ if the  subindex inequality is satisfied and $0$ otherwise.

\begin{theorem}\cite{MRS}\label{embedding thm orth1}
The map $S\mapsto L^-(L^+)^{t}$  defines
an algebra embedding of
$\mathrm{U}_{q}^{\tw}(\mathfrak{o}_{N}) \longrightarrow \mathrm{U}_{q} (\gl_{N}) $.
The monomials
\begin{equation}
\begin{split}
 s_{21}^{k_{21}}s_{31}^{k_{31}}s_{32}^{k_{32}}\cdots s_{N1}^{k_{N1}}s_{N2}^{k_{N2}}\cdots s_{N,N-1}^{k_{N,N-1}}\\
\end{split}
\end{equation}
form a basis of  the algebra $\mathrm{U}_{q}^{\tw}(\mathfrak{o}_{N})$,
where $k_{ij}$ are nonnegative integers.
\end{theorem}

Regarding  $\mathrm{U}_{q}^{\tw}(\mathfrak{o}_{N})$ as a subalgebra of $\mathrm{U}_{q} (\gl_{N})$,
we introduce another matrix $\overline{S}$ by
 \begin{equation}
  \overline{S}=L^+{L^-}^t,
  \end{equation}
then
\begin{equation}
  \overline{S}=1-q+qS^{t}.
  \end{equation}
In terms of the matrix elements, $\overline{s}_{ij}=qs_{ji}$ for $i<j$.
So, the elements $\overline{s}_{ij}$ belong to the subalgebra $\mathrm{U}_{q}^{\tw}(\mathfrak{o}_{N})$.

The subalgebra $\mathrm{U}_{q}^{\tw}(\mathfrak{o}_{N})$ is a coideal of $\mathrm{U}_q(\mathfrak{gl}_N)$, as the
image of the generator $s_{ij}$ under the coproduct is given by
\begin{equation}
\Delta (s_{ij})=\sum_{k,l=1}^{N}l^-_{ik}l^+_{jl}\otimes s_{kl}.
\end{equation}

\subsection{Symplectic case}

Let $G$ be the $2n\times 2n$ matrix $G$ defined by
 \begin{equation}
G=
q\sum_{k=1}^{n}E_{2k-1,2k}-\sum_{k=1}^{n}E_{2k,2k-1}.
 \end{equation}

For $1\leq k\leq n$, we denote
\begin{equation}
 (2k-1)'=2k, \qquad  (2k)'=2k-1.
\end{equation}

The {\it twisted quantum enveloping algebra}
$\mathrm{U}_{q}^{\tw}(\mathfrak{sp}_{N})$
is an unital algebra
generated by $s_{ij},1\leq i,j\leq N$ and the elements  $s_{ii'}^{-1}, i=1,3, \ldots, 2n-1$,
 subject to the reflection relation
 \begin{equation}\label{reflection rel}
 RS_1R^{t}S_2=S_2 R^{t}S_1R,
\end{equation}
and in addition the following relations
\begin{align} \label{q-sp}
&s_{ij}=0  \text{ for }  i<j  \text{ with }  j\neq i', \\
&s_{ii'}s_{ii'}^{-1}=s_{ii'}^{-1}s_{ii'}=1,\ i=1,3,\ \ldots,\ 2n-1, \\
&s_{i'i'}s_{ii}-q^{2}s_{i'i}s_{ii'}=q^{3},\ i=1,3,\ \ldots,\ 2n-1.
 \end{align}

 We still use the symbol $s_{ij}$ for the generators of $\mathrm{U}_{q}^{\tw}(\mathfrak{sp}_{N})$ since they
 also satisfy the reflection equation \eqref{reflection rel} in both cases. It will be clear from the context whether they
 are the generators of the orthogonal algebra or the symplectic algebra.

\begin{theorem}\cite{MRS}\label{embedding thm sp1}
The map $S\to L^-G(L^+)^{t}$  defines
an algebra embedding of
$\mathrm{U}_{q}^{\tw}(\g_{N}) \longrightarrow \mathrm{U}_{q} (\gl_{N}) $.
The monomials
\begin{equation}
\vec{\prod\limits_{i=1,3,\ldots,2n-1}}s_{i1}^{k_{i1}}s_{i2}^{k_{i2}}\cdots s_{ii'}^{k_{ii'}},s_{i'i'}^{k_{i'i'}},s_{i'1}^{k_{i'1}}\cdots s_{i,i-2}^{k_{i',i'-2}}
\end{equation}
form a basis of  the algebra $\mathrm{U}_{q}^{\tw}(\mathfrak{sp}_{N})$,
where   $k_{i'i'}$ with $i=1,3, \ldots, 2n-1$   are arbitrary integers and the remaining exponents  $k_{ij}$ are any nonnegative integers.
\end{theorem}

Regarding $\mathrm{U}_{q}^{\tw}(\mathfrak{sp}_{N})$ as subalgebra of $\mathrm{U}_{q} (\gl_{N})$,
 we introduce another matrix $\overline{S}$ by
 \begin{equation}
  \overline{S}=L^+G{L^-}^t.
  \end{equation}

It is easy to see the following relations between the matrix elements of $S$ and $\overline{S}$: for any $i=1,$ 3, $\ldots, 2n-1$
\begin{align}
  &\overline{s}_{ii}=-q^{-2}s_{ii},\ \overline{s}_{i'i'}=-q^{-2}s_{i'i'},\\
 & \overline{s}_{i'i}=-q^{-1}s_{ii'},\ \overline{s}_{ii'}=-q^{-1}s_{i'i}+(1-q^{-2})s_{ii'},
\end{align}
and for any $ i<j,\ j\neq i'$, 
\begin{equation}
  \overline{s}_{ij}=-q^{-1}s_{ji}.
\end{equation}
Thus, the elements $\overline{s}_{ij}$ belong to the subalgebra $\mathrm{U}_{q}^{\tw}(\mathfrak{sp}_{N})$.

The algebra $\mathrm{U}_{q}^{\tw} (\mathfrak{sp}_{N})$ is another coideal of $\mathrm{U}_{q}(\mathfrak{gl}_{N})$,
as
the coproduct of the generators are given by 
\begin{align}
&\Delta (s_{ij})=\sum_{k,l=1}^{N}l^-_{ik}l^+_{jl}\otimes s_{kl},\\
& \Delta (s_{ii'}^{-1})= l^-_{i'i'}l^+_{ii}\otimes s_{ii'}^{-1}.
\end{align}

\section{Quantum affine algebra $\mathrm{U}_q(\widehat{\gl}_N)$}

The $R$-matrix $R$ is invertible. Let $P=\sum_{ij}E_{ij}\otimes E_{ji}$, then $PR^{-1}P$ becomes another $R$-matrix, explicitly
\begin{equation}
  PR^{-1}P=\qin\sum_{i}E_{ii}\ot E_{ii}+\sum_{i\ne j}E_{ii}\ot E_{jj}
-(q-\qin)\sum_{i> j}E_{ij}\ot E_{ji}.
 \end{equation}
For any two variables $u, v$, we introduce the $R$-matrix $R(u,v)=uPR^{-1}P-v R$:
\begin{equation}
  \begin{split}
    R(u,v)=
(u-v)\sum_{i\neq j}E_{ii}\otimes E_{jj}+(q^{-1}u-qv)\sum_{i}E_{ii}\otimes E_{ii}\
 \\
 +(q^{-1}-q)u\displaystyle \sum_{i>j}E_{ij}\otimes E_{ji}+(q^{-1}-q)v\sum_{i<j}E_{ij}\otimes E_{ji},
 \end{split}
  \end{equation}
which satisfies the spectral parameter dependent Yang-Baxter equation
 \begin{equation}\label{YBE}
  \begin{split}
 R_{12}(u, v)R_{13}(u, w)R_{23}(v, w)=R_{23}(v, w)R_{13}(u, w)R_{12}(u, v)
 \end{split}
  \end{equation}
where both sides take values in $\mathrm{End}(\mathbb{C}^{N})\otimes \mathrm{End}(\mathbb{C}^{N}) \otimes \mathrm{End}(\mathbb{C}^{N})$ and
the subindices indicate the copies of End $\mathbb{C}^{N}$, e.g. $R_{12}(u,v)=R(u,v)\otimes 1$ etc.

The matrix $R(u,v) $ is invertible and
\begin{equation}\label{inverse of R}
  R(u,v)R'(u,v)=(qu-q^{-1}v)(q^{-1}u-qv)1\ot 1.
\end{equation}
where $R'(u,v)$ is obtained from $R(u,v)$ by replacing $q$ with  $q^{-1}$.

We also need the normalized $R$ matrix
\begin{equation}
  \wt R(x)=\frac{R(x,1)}{q^{-1}x-q}, \qquad
  R(x)= f(x)\wt R(x)
\end{equation}
while the formal power series
\begin{equation}
  f(x)=1+\sum_{k=1}^{\infty}f_kx^k,\qquad f_k=f_k(q),
\end{equation}
is uniquely determined
by the relation
\begin{equation}\label{fx relation}
 f(xq^{2N})=f(x)  \frac{(1-xq^2)  (1-xq^{2N-2})}{(1-x)  (1-xq^{2N})}.
\end{equation}

The  $R$-matrix  $R(x)$ satisfies the following {\it crossing symmetry relations} \cite{FR}:
\begin{equation}\label{cross rel}
  \begin{split}
  R_{12}^{-1}(x)^{t_2}D_2R_{12}^{t_2}(xq^{2N})=D_2,\qquad
  R_{12}^{t_1}(xq^{2N}) D_1 R_{12}^{-1}(x)^{t_1}=D_1,
 \end{split}
  \end{equation}
  where $D$ is the diagonal matrix $\diag(q^{N-1},q^{N-3},\cdots,q^{1-N})$.

The {\it quantum affine algebra} $\mathrm{U}_q(\widehat{\gl}_N)$
 is generated by ${l^{\pm}_{ij}}^{(r)}$  where $1\leq i, j\leq N$ and $r$ runs over nonnegative integers. Let $L^{\pm}(u)=(l_{ij}^{\pm}(u))$ be the matrix
 \begin{equation}
  \begin{split}
L^{\pm}(u)= \sum_{i,j=1}^{N}l^{\pm}_{ij}(u)\otimes E_{ij},
 \end{split}
  \end{equation}
where $l^{\pm}_{ij}(u)$  are formal series in $u^{\pm1}$ respectively,
 \begin{equation}
  \begin{split}
 l^{\pm}_{ij}(u)= \sum_{r=0}^{\infty} {l^{\pm}_{ij}}^{(r)}u^{\pm r}.
 \end{split}
  \end{equation}
The defining relations are
 \begin{equation}\label{e:defn}
  \begin{split}
{l^-_{ij}}^{(0)}={l^+_{ji}}^{(0)}=0,\ 1\leq i<j\leq N,  \\
l{^-_{ii}}^{(0)} {l^+_{ii}}^{(0)}={l^+_{ii}}^{(0)}{l^-_{ii}}^{(0)}=1,\ 1\leq i\leq N, \\
R(u/v)L^{\pm}_{1} (u) L^{\pm}_{2} (v)= L^{\pm}_{2} (v) L^{\pm}_{1} (u)R(u/v),\\
R(u q^{-c} /v) L^{+}_{1} (u) L^{-}_{2} (v)= L^{-}_{2} (v) L^{+}_{1} (u)R(u q^c  /v).\\
 \end{split}
  \end{equation}

The quantum enveloping algebra $\mathrm{U}_{q}(\gl_{N})$ is a natural  subalgebra of $\mathrm{U}_q(\widehat{\gl}_N)$ defined by the embedding
 \begin{equation}
  \begin{split}
l^{\pm}_{ij}\mapsto {l^{\pm}_{ij}}^{(0)}.
 \end{split}
  \end{equation}
Moreover, there is an algebra homomorphism $\mathrm{U}_q(\widehat{\gl}_N)\rightarrow \mathrm{U}_{q}(\gl_{N})$ called the {\it evaluation homomorphism} defined by
\begin{equation}
  \begin{split}
\quad L^+(u)\mapsto L^+-uL^{-},L^-(u)\mapsto L^--u^{-1} L^+ , q^c\mapsto 1.
 \end{split}
\end{equation}

\subsection{Quantum (affine) determinants}

Consider the  tensor product $\mathrm{U}_{q}(\widehat{\gl}_{N}) \otimes(\mathrm{End}\mathbb{C}^{N})^{\otimes m}$, we have the following relation:
 \begin{align}\notag
& R(u_{1},  \ldots,  u_{m})L^{\pm}_{1} (u_{1})\cdots L^{\pm}_{m}(u_{m})\\
& \qquad =L^{\pm}_{m}(u_{m})\cdots  L^{\pm}_{1} (u_{1})R(u_{1}, \ldots, u_{m})
 \end{align}
where
 \begin{equation}
 R(u_{1}, \ \ldots,\ u_{m})=\overset{\rightarrow}{\prod}_{1\leq i<j\leq m}R_{ij}(u_{i},u_{j}),
 \end{equation}
and the product is taken in the lexicographical order on the pairs $(i, j)$.

Consider the $q$-permutation operator $P^{q}\in \mathrm{End}(\mathbb{C}^{N}\otimes \mathbb{C}^{N})$ defined by
\begin{equation}
P^{q}= \sum_{i}E_{ii}\otimes E_{ii}+q\sum_{i>j}E_{ij}\otimes E_{ji}+q^{-1}\sum_{i<j}E_{ij}\otimes E_{ji}.
 \end{equation}
The symmetric group $\mathfrak{S}_{m}$ acts on the space $(\mathbb{C}^{N})^{\otimes m}$ via
$s_{i}\mapsto P_{s_{i}}^{q} :=P_{i,i+1}^{q}$ far $i=1, \ldots, m-1$, where $s_{i}$ denotes
the transposition $(i,\ i+1)$ . If $\sigma=s_{i_{1}}\cdots s_{i_{l}}$ is a reduced decomposition of
an element $\sigma\in \mathfrak{S}_{m}$ we set $P_{\sigma}^{q}= P_{s_{i_{1}}}^{q}\cdots P_{s_{i_{l}}}^{q}$.
The $q$-antisymmetrizer is then defined by
\begin{equation}
A_{m}^q =\frac{1}{m!} \sum_{\sigma\in \mathfrak{S}_{m}}\mathrm{sgn}\sigma\cdot P_{\sigma}^{q}.
 \end{equation}
It is known that the $q$-antisymmetrizer can be rewritten in $(\mathrm{End}\mathbb{C}^{N})^{\otimes m}$
\begin{equation}\label{e:q-anti}
R(1,  q^{-2},  \ldots,  q^{-2m+2})=m!\prod_{0\leq i<j\leq m-1}(q^{-2i}-q^{-2j})A_{m}^q .
 \end{equation}

Using the commutation relation \eqref{e:defn} and the Yang-Baxter equation \eqref{YBE} we have that
 \begin{equation}
A_{m}^q L^{\pm}_{1} (u)\cdots L^{\pm}_{m}(q^{-2m+2}u)=L^{\pm}_{m}(q^{-2m+2}u)\cdots L^{\pm}_{1} (u)A_{m}^q.
 \end{equation}
It can be rewritten as
\begin{equation}
  \frac{1}{m!}\sum_{i_{k},j_{k}}{l^{\pm}}_{j_{1} \cdots   j_{m}}^{i_{1} \cdots i_{m}}(u)\otimes E_{i_{1}j_{1}}\otimes\cdots\otimes E_{i_{m}j_{m}}.
\end{equation}
We define the {\it quantum minors for the quantum affine algebra} as the coefficients ${l^{\pm}}_{j_{1} \cdots   j_{m}}^{i_{1} \cdots i_{m}}(u)\in\mathrm{U}_{q}(\gl_N)[[u^{-1}]]$. They generalize the quantum minors of the quantum general linear algebra introduced by
Faddev, Reshetikhin and Tacktajan \cite{FRT}.

Due to the $q$-antisymmetrizer, the quantum minor are given by the following formula. If $i_{1}<\cdots<i_{m}$, then
\begin{equation}
  {l^{\pm}}_{j_{1} \cdots  j_{m}}^{i_{1} \cdots i_{m}}(u)=
  \sum_{\sigma\in \mathfrak{S}_{m}}(-q)^{-l(\sigma)}\cdot {l^{\pm}}_{i_{\sigma(1)}j_{1}}(u)\cdots {l^{\pm}}_{i_{\sigma(m)}j_{m}}(q^{-2m+2}u).
\end{equation}
Subsequently for any $\tau\in \mathfrak{S}_{m}$ we have
\begin{equation}
  {l^{\pm}}_{j_{ 1 } \cdots   j_{ m}}^{i_{\tau(1)} \cdots i_{\tau(m)}}(u)
=(-q)^{l(\tau)}{l^{\pm}}_{j_{1} \cdots  j_{m}}^{i_{1} \cdots i_{m}}(u),
\end{equation}
where $l(\tau)$ denotes the length of the permutation $\tau$. If $j_{1}<\cdots<j_{m}$ (and the $i_{r}$ are arbitrary) then
\begin{equation}
  {l^{\pm}}_{j_{1} \cdots j_{m}}^{i_{1} \cdots i_{m}}(u)= \sum_{\sigma\in \mathfrak{S}_{m}}(-q)^{l(\sigma)}
  \cdot {l^{\pm}}_{i_{m}j_{\sigma(m)}}(q^{-2m+2}u)\cdots {l^{\pm}}_{i_{1}j_{\sigma(1)}}(u)
\end{equation}
and for any $\tau\in \mathfrak{S}_{m}$ we have
\begin{equation}
  {l^{\pm}}_{j_{\tau(1)} \cdots  j_{\tau(m)}}^{i_{1} \cdots i_{m}}(u)
=(-q)^{-l(\tau)}{l^{\pm}}_{j_{1} \cdots  j_{m}}^{i_{1} \cdots i_{m}}(u).
\end{equation}

The {\it quantum determinant} of the matrices ${L^{\pm}} (u) $ are defined as the $N$-minor:
\begin{equation}
  {\det}_q {L^{\pm}} (u)={l^{\pm}}^{1\cdots N}_{1\cdots N}(u),
\end{equation}
which satisfies the following relation:
\begin{equation}
\begin{aligned}
  A_{N}^q L^{\pm}_{1} (u)\cdots L^{\pm}_{N}(q^{2-2N}u)&=L^{\pm}_{N}(q^{2-2N}u)\cdots L^{\pm}_{1} (u)A_{N}^{q}\\
  &={\det}_q {L^{\pm}} (u)A_{N}^{q}.
   \end{aligned}
\end{equation}

Suppose that $I=\{i_1<i_2<\cdots<i_m\}$, $J=\{j_1<j_2<\cdots<j_m\}$,
we denote ${l^{\pm}}_{j_{1} \cdots   j_{m}}^{i_{1} \cdots i_{m}}(u)$ by ${\det}_q\left(L^{\pm}(u)_J^{I}\right)$.
If $I=J$, we denote ${\det}_q\left(L^{\pm}(u)_I\right)={\det}_q\left(L^{\pm}(u)_I^{I}\right)$.

We define the comatrix $\widehat L^{\pm}(u) $  by
\begin{equation}
  \begin{split}
    \widehat L^{\pm}(u) L^{\pm}(q^{2-2N}u)={\det}_{q}(L^{\pm}(u))I.\\
  \end{split}
\end{equation}

\begin{proposition}\label{qdet comatrix}
The matrix elements  of $\widehat L^{\pm}_{ij}(u)$ are given by
\begin{equation}\label{comatrix formula}
  \begin{split}
    \widehat L^{\pm}_{ij}(u)=(-q)^{j-i} {l^{\pm}}^{1,\cdots,\hat j\cdots,N}_{1,\cdot \hat i,\cdots ,N}(u).
  \end{split}
\end{equation}
Moreover, we have the relations
\begin{equation}
  \begin{split}
  { L^{\pm}} (u) ^t  D   {\widehat L^{\pm}}(q^{-2}u)^t D^{-1}={\det}_q(L^{\pm}(u))I.
  \end{split}
\end{equation}
\end{proposition}

\begin{proof}
Multiplying $ \left( L^{\pm}_N(q^{2-2N}u)\right)^{-1}$ from the right of the formulas
\begin{equation}
A_N^q L^{\pm}_1(u)L^{\pm}_2(q^{-2}u)\dots L^{\pm}_N(q^{2-2N}u)
=A_N^q {\det}_q(L^{\pm}(u)).
\end{equation}
We get that
\begin{equation}
  A_N^q L^{\pm}_1(u)L^{\pm}_2(q^{-2}u)\dots L^{\pm}_{N-1}(q^{4-2N}u)=A_N^q \widehat L^{\pm}_{N}(u).
\end{equation}
Applying both sides to the vector
\begin{equation}
e_1\otimes\cdots\hat{e}_{i}\otimes e_{N}\otimes e_j
\end{equation}
and comparing the coefficients of
$e_1\otimes\cdots \otimes e_{N} $ we get the equation \eqref{comatrix formula}.

Denote by $A^{q}_{\{2,\ldots,N\}}$  the $q$-antisymmetrizer over the copies of $End(\mathbb{C}^N)$ labeled by $\{2,\ldots,N\}$.
Then $A_{N}^q=A_{N}^q A^{q}_{\{2,\ldots,N\}}$ and
\begin{equation}
A_N^q L^{\pm}_1(u)A^{q}_{\{2,\ldots,N\}}L^{\pm}_2(q^{-2}u)\dots L^{\pm}_N(q^{2-2N}u)
=A_N^q {\det}_q(L^{\pm}(u)).
\end{equation}
Applying both sides to the vector
\begin{equation}
e_i \otimes e_1\otimes\cdots\hat{e_{j}}\otimes e_{N},
\end{equation}
we get that
\begin{equation}
  \begin{split}
\sum_{k=1}^N (-q)^{j-k}l^{\pm}_{ki} (u) {l^{\pm}}^{1,\cdots,\hat k\cdots,N}_{1,\cdot \hat j,\cdots ,N}(q^{-2}u)=\delta_{ij}{\det}_q(L^{\pm}(u)).
 \end{split}
  \end{equation}
  It can be written as
  \begin{equation}
  \begin{split}
    {L^{\pm}}(u)^t D \left(\widehat L ^{\pm} (q^{-2}u)\right)^t D^{-1}={\det}_q(L^{\pm}(u)).
 \end{split}
  \end{equation}
\end{proof}

\subsection{Minor identities for quantum determinants}
Multiplying the $RTT$ relation  from both sides consecutively by the inverses to $R(u,v)$, $L^{\pm}_{1} (u) $ and $L^{\pm}_{2} (v)$
, then using the relation \eqref{inverse of R} we have that
\begin{equation}
  \begin{split}
R'(u/v)  L^{\pm}_{1}(u)^{-1} L^{\pm}_{2} (v)^{-1}  =L^{\pm}_{2} (v)^{-1}  L^{\pm}_{1}(u)^{-1} R'(u/v),\\
 \end{split}
  \end{equation}

For any two sets $I=\{i_1,\ldots,i_r\}$, $J=\{j_1,\ldots,j_s\}$  with
$i_1<\cdots<i_r$ and $j_1<\cdots<j_s$, we denote by  $l(I,J)$ the inversion number of the sequence $i_1,\ldots.i_r,j_1,\ldots,j_s$.
Let $A$ be any square matrix of size $N\times N$.
For any subsets $I=\{i_1,\ldots,i_k\},J=\{j_1,\ldots,j_k\}$ of  $[1,N]$,
we denote by $A^{I}_{J}$ the matrix whose
$ab$-th entry is $A_{i_a j_b}$.
For any subsets $I=\{i_1<i_2<\cdots i_{k}\}$,
we denote by $A_I$ the submatrix of $A$ with rows and columns indexed by $I$.
  The following theorem is an analog of Jacobi's ratio theorem for the quantum determinants.
  \begin{theorem}\label{Jacobi thm qdet}
  Let   $I=(i_1<\cdots<i_k)$ and $J=(j_1<\cdots<j_k)$  be two subsets of $[1,N]$ of the same cardinality, and
  $I^{c}=\{i_{k+1}<\cdots<i_{N}\}$ and $J^{c}=\{j_{k+1}<\cdots<j_{N}\}$ be their complements.
Then
\begin{align}\notag
   &(-q)^{-l(I,I^c)} {\det}_{q}(L^{\pm}(u)^I_J)\\
   &= (-q)^{-l(J,J^c)} {\det}_q(L^{\pm}(u))  {\det}_{\qin}\left(({L^{\pm}(q^{2-2N}u)^{-1}})^{J^c}_{I^c}\right).
  \end{align}
  \end{theorem}
\begin{proof}
    Multiplying $ \left( L^{\pm}_N(q^{2-2N}u)\right)^{-1},\cdots, \left( L^{\pm}_{k+1}(q^{-2k}u)\right)^{-1}$ on the right of the formula
    \begin{equation}
    A_N^q L^{\pm}_1(u)L^{\pm}_2(q^{-2}u)\dots L^{\pm}_N(q^{2-2N}u)
    =A_N^q {\det}_q(L^{\pm}(u)),
    \end{equation}
we get that
   \begin{equation}
   \begin{split}
  &A_N^q    L^{\pm}_1(u)L^{\pm}_2(q^{-2}u)\dots L^{\pm}_k(q^{2-2k}u)\\
    &=
    {\det}_q(L^{\pm}(u))  A_N^q
    {L^{\pm}_{N}(q^{2-2N}u)}^{-1}  \cdots {L^{\pm}_{k+1}(q^{-2k}u)}^{-1}.
  \end{split}
  \end{equation}

  Applying both sides to the vector $e_{j_{1}}\otimes \dots e_{j_k}\otimes e_{i_N}\otimes\dots \otimes e_{i_{k+1}}$ and comparing the
  coefficient of $e_{1}\otimes e_{2}\otimes  \dots \otimes e_{N}$, we obtain that
  \begin{align}\notag
   &(-q)^{-l(I,I^c)} {\det}_{q}(L^{\pm}(u)^I_J)\\
   &=  (-q)^{-l(J,J^c)} {\det}_q(L^{\pm}(u))  {\det}_{\qin}\left(({L^{\pm}(q^{2-2N}u)^{-1}})^{J^c}_{I^c}\right).
  \end{align}
  \end{proof}

  As a special case with $I=J=\emptyset$, we have the following corollary.
  \begin{corollary}\label{inverse det}
    \begin{equation}
      \begin{split}
 {\det}_q(L^{\pm}(u))  {\det}_{\qin}({L^{\pm}(q^{2-2N}u)^{-1}})=1.
      \end{split}
      \end{equation}
   \end{corollary}

The following is an analogue of Schur's complement theorem.

\begin{theorem}\label{sdet Schur's complement theorem} Let
\begin{equation}
  L^{\pm}(u)=
  \begin{pmatrix}
      L^{\pm}_{11}(u) &L^{\pm}_{12}(u)\\
      L^{\pm}_{21}(u) &L^{\pm}_{22}(u)
    \end{pmatrix}
\end{equation}
be the block matrix such that
$L^{\pm}_{11}(u)$ and $L^{\pm}_{22}(u)$ are submatrices of sizes $k\times k$ and $(N-k)\times (N-k)$ respectively.
Then
\begin{align*}
  &{\det}_q  (L^{\pm}(u))\\
&={\det}_q(L^{\pm}_{11}(u)) {\det}_{q}\left(L^{\pm}_{22}(q^{-2k}u)-L^{\pm}_{21}(q^{-2k}u)L^{\pm}_{11}(q^{-2k}u)^{-1}L^{\pm}_{12}(q^{-2k}u)\right)\\
&={\det}_q(L^{\pm}_{22}(u))  {\det}_q(L^{\pm}_{22}(q^{2(k-N)}u)-L^{\pm}_{21}(q^{2(k-N)}u)L^{\pm}_{11}(q^{2(k-N)}u)^{-1}L^{\pm}_{12}(q^{2(k-N)}u)).
  \end{align*}
\end{theorem}

  \begin{proof}
We denote by $X^{\pm}(u)$ the inverse of $L^{\pm}(u)$  and write it as
\begin{equation}
  X^{\pm}(u)=
  \begin{pmatrix}
      X^{\pm}_{11}(u)&X^{\pm}_{12}(u)\\
      X^{\pm}_{21}(u) &X^{\pm}_{22}(u)
    \end{pmatrix},
\end{equation}
then
\begin{equation}
  \begin{split}
    X^{\pm}_{11}(u)= \left(L^{\pm}_{11}(u)-L^{\pm}_{12}(u)L^{\pm}_{22}(u)^{-1}L^{\pm}_{21}(u)\right) ^{-1},  \\
    X^{\pm}_{22}(u)=  \left(L^{\pm}_{22}(u)-L^{\pm}_{21}(u)L^{\pm}_{11}(u)^{-1}L^{\pm}_{12}(u)\right) ^{-1} .
  \end{split}
\end{equation}

It follows from Theorem \ref{Jacobi thm qdet} that
\begin{equation}\label{Schur complement eq1}
  \begin{split}
    {\det}_q(L^{\pm}_{11}(u))={\det}_q(L^{\pm}(u)){\det}_{q^{-1}}\left(X_{22}^{\pm}(q^{2-2N}u)\right).
  \end{split}
\end{equation}

Since the matrix $
L^{\pm}_{22}(u)-L^{\pm}_{21}(u)L^{\pm}_{11}(u)^{-1}L^{\pm}_{12}(u)$ is the inverse of $X^{\pm}_{22}(u)$, it satisfies the
$RTT$ relations.
By Corollary \ref{inverse det},
\begin{equation}\label{Schur complement eq2}
  \begin{aligned}
   & {\det}_{q^{-1}}(X^{\pm}_{22}(u)) ^{-1}\\
   &=    {\det}_{q}\left(L^{\pm}_{22}(q^{-2k}u)-L^{\pm}_{21}(q^{-2k}u)L^{\pm}_{11}(q^{-2k}u)^{-1}L^{\pm}_{12}(q^{-2k}u)\right).
  \end{aligned}
\end{equation}
Combing \eqref{Schur complement eq1} and \eqref{Schur complement eq2},
we obtain the first equation. The second equation can be proved similarly.
\end{proof}

  Using Jacobi's theorem we obtain the following analog of Cayley's complementary identity for quantum determinants.
  \begin{theorem}\label{skly cayley thm }
  Suppose a minor identity for quantum determinants is given:
  \begin{equation}
  \sum_{i=1}^{k}b_i \prod_{j=1}^{m_i}
  {\det}_{q}\left((L^{\epsilon_{ij}}(u))^{I_{ij}}_{J_{ij}} \right)=0,
  \end{equation}
  where $I_{ij} $ and $J_{ij} $ are subsets of $[1,N]$ ,
  $\epsilon_{ij}=+$ or $-$
  and $b_i\in \mathbb C(q)$.
  Then the following identity holds
  \begin{equation}
  \sum_{i=1}^{k} b_i'\prod_{j=1}^{m_i} (-q)^{l(I_{ij}^c,I_{ij})-l(J_{ij}^c,J_{ij})}
  {\det}_q(L^{\epsilon_{ij}}(u))^{-1}
  {\det}_{q} (L^{\epsilon_{ij}}(u))_{I_{ij}^c}^{J_{ij}^c})=0,
  \end{equation}
  where $b_i'$ is obtained from $b_i$ by changing $q$ by $q^{-1}$.
  \end{theorem}

  \begin{proof}
  The matrix $L^{\pm}(u)^{-1}$ satisfies the $q^{-1}$-$RTT$ relations. Applying the minor identity to $L^{\pm}(u)^{-1}$ we get that
  \begin{equation}
    \sum_{i=1}^{k}b_i' \prod_{j=1}^{m_i}
    {\det}_{q^{-1}}\left((L^{\epsilon_{ij}}(u)^{-1})^{I_{ij}}_{J_{ij}} \right)=0.
  \end{equation}
  It follows from Theorem \ref{Jacobi thm qdet} that
  \begin{equation}
    \begin{aligned}
    &{\det}_{q^{-1}}\left((L^{\epsilon_{ij}}(u)^{-1})^{I_{ij}}_{J_{ij}} \right)\\
    &=(-q)^{l(I_{ij}^c,I_{ij})-l(J_{ij}^c,J_{ij})}{\det}_q(L^{\epsilon_{ij}}(q^{2N-2}u))^{-1}
    {\det}_{q} (L^{\epsilon_{ij}}(q^{2N-2}u)_{I_{ij}^c}^{J_{ij}^c}).
  \end{aligned}
  \end{equation}
The proof is completed  by replacing $u$ with $q^{2-2N}u$.

  \end{proof}

  The following theorem is an analog of Muir's law for the quantum determinants.
  \begin{theorem}\label{qdet muir law}
  Suppose a quantum minor determinant identity is given:
  \begin{equation}
    \sum_{i=1}^{k}b_i \prod_{j=1}^{m_i}
    {\det}_{q}\left((L^{\epsilon_{ij}}(u))^{I_{ij}}_{J_{ij}} \right)=0,
    \end{equation}
  where $I_{ij}'s$ are subsets of $T= \{1,2,\dots,N\}$,
  $\epsilon_{ij}=+$ or $-$
  and $b_i\in \mathbb C(q)$.
  Let $K$ be the set $\{N,\dots, N+M\}$. Then the following identity holds
  \begin{equation}\label{qdet muir law equation}
  \sum_{i=1}^{k} b_i\prod_{j=1}^{m_i}
  {\det}_{q} ((L^{\epsilon_{ij}}(u))_{K})^{-1}
  {\det}_{q} ((L^{\epsilon_{ij}}(u))^{I_{ij}\cup K}_{J_{ij}\cup K})=0.
  \end{equation}
  \end{theorem}
  \begin{proof}
  Applying Cayley's complementary identity for the set $T$, we get that
  \begin{equation}\begin{aligned}
    &\sum_{i=1}^{k} b_i'\prod_{j=1}^{m_i} (-q)^{l(T\setminus I_{ij},I_{ij})-l(T \setminus J_{ij},J_{ij})}{\det}_q(L^{\epsilon_{ij}}(u)_T)^{-1}\\
    & \qquad\cdot{\det}_{q} (L^{\epsilon_{ij}}(u)_{T \setminus  I_{ij}}^{T \setminus J_{ij}})=0.
    \end{aligned}
    \end{equation}

  The equation \eqref{qdet muir law equation} is similarly shown by using Cayley's
  complementary identity for the set $T\cup K$.
  \end{proof}

  The following is an analog of Sylvester's theorem for the quantum determinants given by Hopkins and Molev \cite{HM}. We will give a proof using Muir's Law.
  \begin{theorem}\label{qdet sylvester}
    Let
    $T=\{1,\cdots,N\}$ ,
    $K=\{N+1,\cdots,N+M\}$.
   Then the mapping $l^{\pm}_{ij}(u)\mapsto  {l^{\pm}}^{i,N+1,\cdots,N+M}_{j,N+1,\cdots,N+M}(u)$
   defines an algebra morphism $\mathrm U_q(\gl_N)\rightarrow \mathrm U_q(\gl_{M+N})$. Denote $\tilde{l}^{\pm}_{ij}(u)$ by the image of  $l^{\pm}_{ij}(u)$
   and $\tilde{L}^{\pm}(u)=(\tilde{l}^{\pm}_{ij}(u))$. Then
   \begin{equation}
    \begin{split}
      {\det}_{q}( \tilde{L}^{\pm}(u)) ={\det}_{q}( {L}^{\pm}(u))  \prod_{i=1}^{N-1} {\det}_{q}( {L}^{\pm}(q^{-2i}u)_K).
      \end{split}
    \end{equation}
   \end{theorem}

    \begin{proof}

    It follows from Muir's Law (Theorem \ref{qdet muir law}) and centrality of quantum minors that the map  defines an algebra morphism.
    Applying Muir's law to the equation
    \begin{equation}
    \begin{split}
      {\det}_{q}( {L}^{\pm}(u)_T)=
        \sum_{\sigma\in \mathfrak{S}_{N}}(-q)^{-l(\sigma)}\cdot {l^{\pm}}_{{\sigma(1)}1}(u)\cdots {l^{\pm}}_{{\sigma(N)}N}(q^{2-2N}u)
    \end{split}
    \end{equation}
    we get that
    \begin{align}\notag
        &{\det}_{q}( {L}^{\pm}(u)_K)^{-1}{\det}_{q}( {L}^{\pm}(u))\\
        &={\det}_{q}( \tilde{L}^{\pm}(u)) \prod_{i=0}^{N-1} {\det}_{q}( {L}^{\pm}(q^{-2i}u)_K)^{-1}.
      \end{align}
Consequently we see that
\begin{equation}
    {\det}_{q}( \tilde{L}^{\pm}(u)) ={\det}_{q}( {L}^{\pm}(u))  \prod_{i=1}^{N-1} {\det}_{q}( {L}^{\pm}(q^{-2i}u)_K).
  \end{equation}
\end{proof}

Denote by $H_{m}^q$ the $q$-symmetrizer
\begin{equation}
H_{m}^{q}= \frac{1}{m!} \sum_{\sigma\in \mathfrak{S}_{m}}  P_{\sigma}^{q}.
 \end{equation}

   \begin{lemma} In the quantum affine algebra $U_q(\widehat{\frak{gl}}_N)$
    \begin{equation}
      L^{\pm}_{1} (u)\cdots L^{\pm}_{m}(q^{-2m+2}u)H_{m}^{q}= H_{m}^{q}L^{\pm}_{1} (u)\cdots L^{\pm}_{m}(q^{-2m+2}u)H_{m}^{q}.
       \end{equation}
   \end{lemma}
\begin{proof}
  Since $H_2^q+A_2^q=1$, using the $RTT$ relations we have that
\begin{equation}
  A_{2}^{q} L^{\pm}_{1} (u)  L^{\pm}_{2}(q^{-2}u)H_{2}^{q} =0
   \end{equation}

   Then
   \begin{equation}
    \begin{split}
      P^q L^{\pm}_{1} (u)  L^{\pm}_{2}(q^{-2}u)H_{2}^{q} = L^{\pm}_{1} (u)  L^{\pm}_{2}(q^{-2}u)H_{2}^{q}.
    \end{split}
         \end{equation}
The lemma follows from the equation.
\end{proof}

In the following, we give an analog of MacMahon's master theorem for the quantum affine algebra.
 \begin{theorem}In the quantum affine algebra $U_q(\widehat{\frak{gl}}_N)$,
  \begin{align}
  &\sum_{r=0}^k (-1)^r tr_{1,\ldots,k}H_{r}^{q}A^{q}_{\{r+1,\ldots,k\}}
  L^{\pm}_{1} (u)\cdots L^{\pm}_{k}(q^{-2k+2}u)=0,\\
  &\sum_{r=0}^k (-1)^r tr_{1,\ldots,k}A^q_{r} H^{q}_{\{r+1,\ldots,k\}}
  L^{\pm}_{1} (u)\cdots L^{\pm}_{k}(q^{-2k+2}u)=0,
  \end{align}
  where $H^{q}_{\{r+1,\ldots,k\}}$ and $A^{q}_{\{r+1,\ldots,k\}}$
denote the symmetrizer and antisymmetrizer over the copies of $End(\mathbb C^N)$ labeled
by $\{r+1,\ldots,k \}$ respectively.
  \end{theorem}
\begin{proof} We show the following identity:
\begin{equation}\label{qdet trace replacement}
\begin{aligned}
&tr_{1,\ldots,k}H_{r}^{q}A^{q}_{\{r+1,\ldots,k\}}
L^{\pm}_{1} (u)\cdots L^{\pm}_{k}(q^{-2k+2}u)\\
=&tr_{1,\ldots,k}\frac{r(k-r+1)}{k}H_{r}^{q}A^{q}_{\{r,\ldots,k\}}
L^{\pm}_{1} (u)\cdots L^{\pm}_{k}(q^{-2k+2}u)\\
+
&tr_{1,\ldots,k}\frac{(r+1)(k-r)}{k}H_{r+1}^{q}A^{q}_{\{r,\ldots,k\}}
L^{\pm}_{1} (u)\cdots L^{\pm}_{m}(q^{-2m+2}u).
\end{aligned}
\end{equation}

In fact, using the group algebra of $\mathfrak{S}_{k}$, we have that
\begin{align*}
(k-r+1)A^q_{\{r,\ldots,k\}}&=A^q_{\{r+1,\ldots,k\}}-(k-r)A^q_{\{r+1,\ldots,k\}}  P^q_{r,r+1} A^q_{\{r+1,\ldots,k\}},\\
(r+1)H^q_{r+1}&=H_{r}^{q}+ r H^{q}_{r} P^q_{r,r+1} H^{q}_{r}.
\end{align*}
Therefore
\begin{equation}
\begin{aligned}
&tr_{1,\ldots,k} H_{r}^{q} A^q_{\{r+1,\ldots,k\}}  P^q_{(r,r+1)} A^q_{\{r+1,\ldots,k\}}
 L^{\pm}_{1} (u)\cdots L^{\pm}_{k}(q^{-2k+2}u)\\
&=tr_{1,\ldots,k} H_{r}^{q} P^q_{(r,r+1)} A^q_{\{r+1,\ldots,k\}} L^{\pm}_{1} (u)\cdots L^{\pm}_{k}(q^{-2k+2}u) A^q_{\{r+1,\ldots,k\}}  \\
&=tr_{1,\ldots,k}H_{r}^{q} P^q_{(r,r+1)} A^q_{\{r+1,\ldots,k\}}L^{\pm}_{1} (u)\cdots L^{\pm}_{k}(q^{-2k+2}u).
\end{aligned}
\end{equation}
Similarly,
\begin{align}\notag
tr_{1,\ldots,k}H^{q}_{r} P^q_{r,r+1} H^{q}_{r}  A^q_{\{r+1,\ldots,k\}}L^{\pm}_{1} (u)\cdots L^{\pm}_{k}(q^{-2k+2}u)\\
=tr_{1,\ldots,k} H_{r}^{q} P^q_{(r,r+1)} A^q_{\{r+1,\ldots,k\}}L^{\pm}_{1} (u)\cdots L^{\pm}_{k}(q^{-2k+2}u).
\end{align}
These imply equation \eqref{qdet trace replacement}. Therefore the telescoping sum equals to zero.

The second equation can be proved by the same arguments.
\end{proof}

\subsection{Liouville formula}
Now we introduce a family   of central elements for the quantum affine algebra \cite{JLM}. Let
  \begin{equation}
  \begin{split}
z^{\pm}(u)^{-1}= \frac{1}{N}tr\left(L^{\pm}(u) D^{-1} L^{\pm}(q^{-2N}u)^{-1}D\right).
 \end{split}
  \end{equation}

  \begin{theorem}\label{qdet z}\cite{JLM} On the quantum affine algebra $U_q(\widehat{\frak{gl}}_N)$, we have that
   \begin{equation}
  \begin{split}
z^{\pm}(u)^{-1}=\frac{{\det}_{q}(L^{\pm}(u))}{{\det}_{q}(L^{\pm}(q^{-2}u))}.
 \end{split}
  \end{equation}

Moreover, we also have
  \begin{equation}
    \begin{split}
  z^{\pm}(u)^{-1}= \frac{1}{N}tr\left(D^{-1} L^{\pm}(q^{-2N}u)^{-1} D L^{\pm}(u)\right),
   \end{split}
    \end{equation}
  \begin{equation}
    \begin{split}
 z^{\pm}(u)=\frac{1}{N} tr\left( D^{-1}{(L^{\pm}}(u)^t)^{-1}D {L^{\pm}}(q^{-2N}u)^t  \right)\\
   = \frac{1}{N} tr\left(
    {L^{\pm}}(q^{-2N}u)^t D^{-1}{(L^{\pm}}(u)^t)^{-1}D \right).
   \end{split}
    \end{equation}

  \end{theorem}
\begin{proof} We compute that
   \begin{align*}
z^{\pm}(u)^{-1}&= \frac{1}{N}tr\left(L^{\pm}T(u) D^{-1} L^{\pm}(q^{-2N}u)^{-1} D\right)\\
&=\frac{1}{N}tr\left(L^{\pm}(u) D^{-1} \hat L^{\pm} (q^{-2}u)D\right)  {\det}_{q}(L^{\pm} (q^{-2}u))^{-1}\\
&=\frac{{\det}_{q}(L^{\pm}(u))}{{\det}_{q}(L^{\pm}(q^{-2}u))}.
  \end{align*}
By the same argument,
  \begin{equation}
    \begin{split}
 \frac{1}{N}tr\left( D^{-1} L^{\pm}(q^{-2N}u)^{-1} D L^{\pm}(u)\right)
  =\frac{{\det}_{q}(L^{\pm}(u))}{{\det}_{q}(L^{\pm}(q^{-2}u))}.
   \end{split}
    \end{equation}

It follows from Proposition \ref{Skl comatrix} that
\begin{equation}
  \begin{split}
    (L^{\pm} (u)^t)^{-1}=D\widehat L^{\pm}(q^{-2}u)^t D^{-1} {\det}_q(L^{\pm}(u))^{-1},\\
 \end{split}
  \end{equation}
 then
    \begin{equation}
        tr\left( D^{-1}(L^{\pm}(u)^t)^{-1}D L^{\pm}(q^{-2N}u)^t\right)
        =tr\left( \widehat L^{\pm}(q^{-2}u)^t L^{\pm}(q^{-2N}u)^t\right).
      \end{equation}
Using the comatrix, we immediately have that
\begin{equation}
  \begin{split}
 \frac{1}{N} tr\left( D^{-1}(L^{\pm}(u)^t)^{-1}D L^{\pm}(q^{-2N}u)^t\right)
 =  \frac{{\det}_{q}(L^{\pm}(q^{-2}u))}{{\det}_{q}(L^{\pm}(u))}=z^{\pm}(u).
 \end{split}
  \end{equation}
The other equation can be proved similarly.
\end{proof}

Denote $Q=P^{t_1}=P^{t_2}=\sum_{i,j=1}^N E_{ij}\otimes E_{ij}$.
\begin{proposition}\label{center z relation} The following relations hold
   \begin{align}\notag
&Q L^{\pm}_1(u) D_2 (L^{\pm}_2(q^{-2N}u)^{-1})^t D_2^{-1}\\
&= D_2 (L^{\pm}_2(q^{-2N}u)^{-1})^t   D_2^{-1} L^{\pm}_1(u)Q=Qz^{\pm}(u)^{-1},\\ \notag
&Q D_1^{-1} (L^{\pm}_1(u)^t)^{-1}D_1 L^{\pm}_2(q^{-2N}u)\\
&= L^{\pm}_2(q^{-2N}u)D_1^{-1}(L^{\pm}_1(u)^t)^{-1}D_1 Q=Qz^{\pm}(u).
  \end{align}
\end{proposition}
\begin{proof}
Multiplying both sides of the relation
\begin{equation}\label{RLL}
  \begin{split}
R(u,v) L^{\pm}_1(u)L^{\pm}_2(v)= L^{\pm}_2(v)L^{\pm}_1(u)R(u,v)
 \end{split}
  \end{equation}
  by $L^{\pm}_2(v)^{-1}$
 and taking partial transpose with respect to the second copy, we get that
 \begin{equation}
  \begin{split}
R(u,v)^{t_2}(L^{\pm}_2(v)^{-1})^t L^{\pm}_1(u)=  L^{\pm}_1(u) (L^{\pm}_2(v)^{-1})^tR(u,v)^{t_2}.
 \end{split}
  \end{equation}
Multiply the inverse of $R(u,v)^{t_2}$, the above can be rewritten as
  \begin{equation}
  \begin{split}
    (L^{\pm}_2(v)^{-1})^t L^{\pm}_1(u) (R(u,v)^{t_2})^{-1} = (R(u,v)^{t_2})^{-1} L^{\pm}_1(u) (L^{\pm}_2(v)^{-1})^t.
 \end{split}
  \end{equation}
The take $v=q^{-2N}u$,
  \begin{equation}
  \begin{split}
    D_2 (L^{\pm}_2(q^{-2N}u)^{-1})^t L^{\pm}_1(u)  D_2^{-1}Q  \\
 =  Q D_2  L^{\pm}_1(u) (L^{\pm}_2(q^{-2N}u)^{-1})^t D_2^{-1}.
 \end{split}
  \end{equation}
Note that for any $X$, $QX_1=QX_2^t$, and $X_1Q=X_2^tQ$. Therefore
\begin{equation}
\begin{split}
 Q L^{\pm}_1(u) D_2  (L^{\pm}_2(q^{-2N}u)^{-1})^t   D_2^{-1}
 =Q L^{\pm}_2(u) ^t D_2  (L^{\pm}_2(q^{-2N}u)^{-1})^t   D_2^{-1}\\
 =Q L^{\pm}_2(u)^t D_2  \widehat L^{\pm} _2(q^{-2}u)^t   D_2^{-1}({\det}_{q}L^{\pm}(q^{-2}u))^{-1}\\
 \end{split}
  \end{equation}
  which is $z^{\pm}(u)^{-1}Q$ by Proposition \ref{qdet comatrix}.

Taking transposition of first copy
in the equation \eqref{RLL}
we get that
  \begin{equation}
    L^{\pm}_1(u)^t R^t(u,v)L^{\pm}_2(v)= L^{\pm}_2(v)R^t(u,v)L^{\pm}_1(u)^t,
  \end{equation}
where we abbreviate $R^{t_1}(u, v)$ as $R^{t}(u, v)$.
  Taking inverses of  $L^{\pm}_1(u)^t$ and $R^t(u,v)$, the relation becomes
    \begin{equation}
  \begin{split}
    L^{\pm}_2(v)(L^{\pm}_1(u)^t)^{-1}R^t(u,v) ^{-1} =  R^t(u,v)^{-1} (L^{\pm}_1(u)^t)^{-1} L^{\pm}_2(v) .
 \end{split}
  \end{equation}
  Using the crossing symmetry,
  we have
\begin{align}\notag
    &L^{\pm}_2(v) (L^{\pm}_1(u)^t)^{-1}D_1 R^{-1}(q^{-2N}u,v) ^{t} D_1^{-1}\\
    &=D_1   R^{-1}(q^{-2N}u,v) ^{t} D_1^{-1} (L^{\pm}_1(u)^t)^{-1}L^{\pm}_2(v).
  \end{align}
 Letting $v=q^{-2N}u$, we rewrite it as
\begin{align}\notag
    &L^{\pm}_2(q^{-2N}u) D_1^{-1}  (L^{\pm}_1(u)^t)^{-1}  D_1 Q\\
    &= Q D_1^{-1} (L^{\pm}_1(u)^t)^{-1}D_1 L^{\pm}_2(q^{-2N}u)\\
    &= QL^{\pm}_1(q^{-2N}u)^t D_1^{-1} (L^{\pm}_1(u)^t)^{-1}D_1
  \end{align}
  which is $z^{\pm}(u)Q$ by Theorem \ref{qdet z}.
\end{proof}

\section{Coideal subalgebras of $\mathrm{U}_q(\widehat{\gl}_N)$}
\subsection{Orthogonal twisted $q$-Yangians}
Consider the element $R^{t}(u, v) :=R^{t_{1}}(u,v)$ obtained from $R(u,v)$: 
   \begin{equation}
  \begin{split}
R^{t}(u,v)=(u-v)\sum_{i\neq j}E_{ii}\otimes E_{jj}+(q^{-1}u-qv)\sum_{i}E_{ii}\otimes E_{ii}\\
+(q^{-1}-q)u\displaystyle \sum_{i>j}E_{ji}\otimes E_{ji}+(q^{-1}-q)v\sum_{i<j}E_{ji}\otimes E_{ji}.
 \end{split}
  \end{equation}

We define the {\it twisted $q$-Yangian} $\mathrm{Y}_q^{\tw}({\mathfrak{o}}_N)$ as the associative algebra generated by $s_{ij}^{(r)}$,
$1\leq i, j\leq N$ and $r\in\mathbb Z_+$ subject to the following relations.
\begin{equation}
  \begin{split}
s_{ij}^{(0)}=0,\ 1\leq i<j\leq N,\qquad
s_{ii}^{(0)}=1,\ 1\leq i\leq N,\\
R(u/v)S_{1}(u)R^{t}(1/uv)S_{2}(v)=S_{2}(v)R^{t}(1/uv)S_{1}(u)R(u/ v),
 \end{split}
  \end{equation}
where
\begin{equation}
\begin{split}
s_{ij}(u)= \sum_{r=0}^{\infty}s_{ij}^{(r)}u^{-r}.
\end{split}
\end{equation}

\begin{theorem}\cite{MRS}\label{embedding thm orth}
The map $S(u)\mapsto L^-(uq^{-c}) L^{+}(u^{-1})^{t}$  defines
an algebra embedding of
$\mathrm{Y}_{q}^{\tw}({ {\mathfrak{o}}}_{N}) \longrightarrow \mathrm{U}_{q} (\widehat{\gl}_{N}) $.
The ordered monomials in the generators   constitute a basis
of $\mathrm{Y}_{q}^{\tw}({\mathfrak{o}}_{N})$.
\end{theorem}

As a subalgebra, $\mathrm{Y}_{q}^{\tw}({ {\mathfrak{o}}}_{N})$ becomes a coideal of $\mathrm{U}_{q} (\widehat{\gl}_{N}) $, as
\begin{equation}
\begin{split}
 \Delta(s_{ij}(u))=\sum_{k,l=1}^{N}l^-_{ik}(u)l^+_{jl}(u^{-1})\otimes s_{kl}(u).
\end{split}
\end{equation}

 Clearly the map $s_{ij}\mapsto s_{ij}^{(0)}$   defines an embedding $\mathrm{U}_{q}^{\tw}( \mathfrak{o}_{N}) \rightarrow \mathrm{Y}_{q}^{\tw}({\mathfrak{o}}_{N})$.
On the other hand, the mapping
\begin{equation}
S(u)\mapsto S+q^{-1}u^{-1} \overline{S}
\end{equation}
defines an algebra homomorphism
$\mathrm{Y}_{q}^{\tw}( \mathfrak{o}_{N})\rightarrow \mathrm{U}_{q}^{\tw}( \mathfrak{o}_{N})$.

We introduce the matrices $\overline{S}(u)=(\overline{s}_{ij}(u))$ by
\begin{equation}
  \overline{S}(u)= L^+(uq^{-c}) {L^-(u^{-1})}^{t} ,
\end{equation}
then one has the following relations,
\begin{equation}
  \begin{split}
     (uq-u^{-1}q^{-1})\overline{s}_{ij}(u)=(uq^{\delta_{ij}}-u^{-1}q^{-\delta_{ij}})s_{ji}(u^{-1})\\
  +(q-q^{-1})(u\delta_{j<i}+u^{-1}\delta_{i<j})s_{ij}(u^{-1}),
\end{split}
\end{equation}
this implies that the coefficients of the series $\overline{s}_{ij}(u)$ generate $\mathrm{Y}_{q}^{\tw}( \mathfrak{o}_{N})$.
\subsection{Symplectic twisted -Yangians}

The {\it Symplectic twisted} $q$-Yangians
$\mathrm{Y}_{q}^{\tw}(\mathfrak{sp}_{N})$ is generated by $s_{ij}^{(r)}$,
with $1\leq i, j\leq N$ and $r$ runs over nonnegative integers,
and $(s_{ii'}^{(0)})^{-1}$ with $i=1,3, \ldots, 2n-1$.
The defining relations are
\begin{equation}
\begin{aligned}
s_{ij}^{(0)}&=0 \text{ for }   i<j \text{ with } j\neq i', \\
s_{i'i'}^{(0)}-q^{2}s_{i'i}^{(0)}s_{ii'}^{(0)}&=q^{3},\text{ for } i=1,3, \ldots, 2n-1 ,\\
s_{ii'}^{(0)}(s_{ii'}^{(0)})^{-1}&=(s_{ii'}^{(0)})^{-1}s_{ii'}^{(0)}=1,\\
R(u/v)S_{1}(u)R ^{t}(1/uv)S_{2}(v)&=S_{2}(v)R ^{t}(1/u v)S_{1}(u)R(u/v),
 \end{aligned}
  \end{equation}
where
\begin{equation}
\begin{split}
s_{ij}(u)=\displaystyle \sum_{r=0}^{\infty}s_{ij}^{(r)}u^{-r}.
\end{split}
\end{equation}

\begin{theorem}\cite{MRS}\label{embedding thm orth}
The map $S(u)\mapsto L^-(uq^{-c}) G L^{+}(u^{-1})^{t}$  defines
an algebra embedding of
$\mathrm{Y}_{q}^{\tw}({\mathfrak{sp}}_{N}) \longrightarrow \mathrm{U}_{q} (\widehat{\gl}_{N}) $.
The ordered monomials in the generators   constitute a basis
of $\mathrm{Y}_{q}^{\tw}({\mathfrak{sp}}_{N})$.
\end{theorem}

Regarding  $\mathrm{Y}_{q}^{\tw}({\mathfrak{sp}}_{N})$ as subalgebra  of $\mathrm{U}_{q} (\widehat{\gl}_{N}) $, we also have
\begin{equation}
\begin{split}
 \Delta(s_{ij}(u))=\sum_{k,l=1}^{N}t_{ik}(u)\overline{t}_{jl}(u^{-1})\otimes s_{kl}(u).
\end{split}
\end{equation}

 The map $s_{ij}\mapsto s_{ij}^{(0)}, s_{ii'}^{-1}\mapsto(s_{ii}^{(0)})^{-1}$   defines an embedding $\mathrm{U}_{q}^{\tw}( \mathfrak{sp}_{N}) \rightarrow \mathrm{Y}_{q}^{\tw}({\mathfrak{sp}}_{N})$. On the other hand,
the mapping
\begin{equation}
S(u)\mapsto S+qu^{-1} \overline{S},\qquad (s_{ii}^{(0)})^{-1} \mapsto s_{ii'}^{-1}
\end{equation}
defines an algebra homomorphism
$\mathrm{Y}_{q}^{\tw}( \mathfrak{sp}_{N})\rightarrow \mathrm{U}_{q}^{\tw}( \mathfrak{sp}_{N})$.

We introduce the matrices $\overline{S}(u)=(\overline{s}_{ij}(u))$ by
\begin{equation}
  \overline{S}(u)= L^+(uq^{-c}) G {L^-(u^{-1})}^{t} ,
\end{equation}
then one has the following relations,
\begin{equation}
  \begin{split}
    (u^{-1}q-uq^{-1})\overline{s}_{ij}(u)=(uq^{\delta_{ij}}-u^{-1}q^{-\delta_{ij}})s_{ji}(u^{-1})\\
    +(q-q^{-1})(u\delta_{i<j}+u^{-1}\delta_{j<i})s_{ij}(u^{-1}),
\end{split}
\end{equation}
which implies that the coefficients of the series $\overline{s}_{ij}(u)$ also generate
$\mathrm{Y}_{q}^{\tw}( \mathfrak{sp}_{N})$.

\subsection{ Sklyanin determinant}
From now on, the symbol $\mathrm{Y}_{q}^{\tw}(\g_N)$ will denote either $\mathrm{Y}_{q}^{\tw}(\mathfrak{o}_{N})$ or $\mathrm{Y}_{q} (\mathfrak{sp}_{N})$ (with $N=2n$ in the latter).
For convenience, we introduce the following $R$-matrix:
\begin{equation}
  \overline{R}(x)=\frac{R(x,1)}{x-1}.
\end{equation}

For any permutation $i_1,i_2,\dots,i_m$ of $1,2,\dots,m$ we denote
\begin{align}\notag
&<S_{i_1}, S_{i_2}, \cdots , S_{i_m}>\\
&=S_{i_1}(u_{i_1}) (\ol R^{t}_{i_1i_2}\cdots \ol R^{t}_{i_1i_m})
S_{i_2}(u_{i_2}) (\ol R^{t}_{i_2 i_3}\cdots \ol R^{t}_{i_2i_m})
\cdots S_{i_m}(u_{i_m}),
\end{align}
where $\ol R^t_{i_ai_b}=\ol R^t_{i_ai_b}(u^{-1}_{i_a}u^{-1}_{i_b})$.
The following relation in the algebra
$\mathrm{Y}_{q}^{\tw}(\g_N)\otimes(\mathrm{End}\mathbb{C}^{N})^{\otimes m}$ follows from the reflection relation
\begin{equation}
\begin{split}
 R(u_{1}, \ldots, u_{m})\langle S_{1}, S_{2}, \cdots , S_{m}\rangle
 =\langle S_{m}, S_{m-1}, \cdots , S_{1}\rangle R(u_{1},\ldots, u_{m}).
 \end{split}
\end{equation}

Specializing $u_{i}=uq^{-2i+2}$, we have that (cf. \eqref{e:q-anti})
\begin{equation}
\begin{split}
A_{m}^{q} \langle S_{1}, S_{2}, \cdots , S_{m}\rangle
=\langle S_{m}, S_{m-1}, \cdots , S_{1}\rangle  A_{m}^{q}.
 \end{split}
\end{equation}

This element can be written as
\begin{equation}
\frac{1}{m!}\sum s^{i_1\cdots i_m}_{j_1\cdots j_m}(u)\otimes e_{i_1 j_1}\otimes \dots \otimes e_{i_mj_m},
\end{equation}
where the sum is taken over all $i_k,j_k\in \{1,2,\cdots, N\}$.
We call $s^{i_1\cdots i_m}_{j_1\cdots j_m}(u)$
the {\it Sklyanin minor} associated to the rows $i_1\cdots i_m$ and the columns  $j_1\cdots j_m$. 

Clearly if $i_k=i_l$ or $j_k=j_l$ for $k\neq l$, then $s^{i_1\cdots i_m}_{j_1\cdots j_m}(u)=0$.
Suppose that $i_1<i_2<\dots<i_m$ and $j_1<j_2<\dots<j_m$. Then the Sklyanin minors 
satisfy the relations:
\begin{equation}
\begin{split}
s^{i_{\sigma(1)}\cdots i_{\sigma(m)}}_{j_{\tau(1)}\cdots j_{\tau(m)}}(u)
=(-q)^{-l(\sigma)-l(\tau)}
s^{i_1\cdots i_m}_{j_1\cdots j_m}(u).
\end{split}
\end{equation}
Note that $s^i_j=s_{ij}(u)$. In particular,
$s^{1\cdots N}_{1\cdots N}(u)$ is called the {\it Sklyanin determinant} and will be denoted by $\sdet_q S(u)$.

The following theorem provides an expression of  $\sdet_q S(u)$ in terms of the quantum determinants.
\begin{theorem}\cite{MRS} We have
\begin{equation}
 \sdet_q S(u)=\gamma_{N}(u){\det}_q L^-(uq^{-c}) {\det}_q L^+(q^{2N-2}u^{-1}),
\end{equation}
where
\begin{equation}
\gamma_N(u)= \left \{
  \begin{aligned}
   &1,   &\text{Case }(\mathfrak{o}_N),\\
    &\frac{q^{n-2}-q^{n}u^{2}}{q^{2n-2}-q^{-2n}u^2}, &\text{Case }(\mathfrak{sp}_N).\\
  \end{aligned}
\right.
\end{equation}

\end{theorem}

We introduce the elements $c_{k}\in \mathrm{Y}_{q}^{\tw}(\mathfrak g_N)$ as the coefficients of
the following series $c(u)$:
 \begin{equation}
c(u)=\gamma_{N}(u)^{-1}  \sdet_q  S(u)=1+ \sum_{k=1}^{\infty}c_{k}u^{-k}.
 \end{equation}

 \begin{corollary}\cite{MRS}
 The coefficients $c_{k}, k\geq 1$
  belong to the center of the algebra $\mathrm{Y}_{q}^{\tw}(\mathfrak g_N).$
  \end{corollary}

\begin{proposition}

The coefficients of the Sklyanin determinants  $\sdet_q(S+q^{-1}u^{-1}\overline{S})$   and
$\sdet_q(S+qu^{-1}\overline{S})$   are central elements in the algebras $\mathrm{U}_{q}^{\tw}(\mathfrak{o} _{N})$ and $\mathrm{U}_{q}^{\tw}(\mathfrak{sp} _{N})$  respectively.

\end{proposition}

We define the auxiliary minor
$\check s^{i_1,\dots,i_m}_{j_1,\dots,j_{m-1},c}(u)$
by
\begin{equation}\begin{split}
&m! A_m ^q \langle S_1,\dots,S_{m-1} \rangle
\ol R^{t}_{1m}( {u_1^{-1}u_m^{-1}})\cdots \ol R^{t}_{m-1,m}( {u_{m-1}^{-1}u_m^{-1}})\\
&=
\sum \check s^{i_1\cdots i_m}_{j_1\cdots j_{m-1},c}(u) \otimes e_{i_1 j_1}\otimes \dots \otimes e_{i_m c},
\end{split}
\end{equation}
where the sum is taken over all $i_k,j_k,c\in \{1,2,\cdots, N\}$.

If $i_k=i_l$ or $j_k=j_l$ for $k\neq l$, then $\check s^{i_1\cdots i_m}_{j_1\cdots j_m-1,c}(u)=0$.
Suppose $i_1<i_2<\dots<i_m$ and  $j_1<j_2<\dots<j_{m-1}$, then
\begin{equation}
\begin{split}
\check s^{i_{\sigma(1)}\cdots i_{\sigma(m)}}_{j_{\tau(1)}\cdots j_{\tau(m-1)},c}(u)
=(-q)^{l(\sigma)-l(\tau)}\check s^{i_1\cdots i_m}_{j_1\cdots j_{m-1},c}(u).
\end{split}
\end{equation}
$\sigma\in S_m, \tau\in S_{m-1}$.

\begin{lemma} For $i_1<i_2<\dots<i_m$ and  $j_1<j_2<\dots<j_{m-1}$,
\begin{equation}
s ^{i_1,\dots,i_m}_{j_1,\dots,j_m}(u)
=\sum_{c=1}^{N}
\check s^{i_1,\dots,i_m}_{j_1,\dots,j_{m-1},c}(u)s_{c j_m}(uq^{2-2m}).
\end{equation}
\end{lemma}
\begin{proof} The identity follows from the formula
\begin{equation}
  \begin{split}
A_m^q \langle S_1,\dots,S_m \rangle
=A_m^q \langle S_1,\dots,S_{m-1} \rangle
\ol R^{t}_{1m}( {u_1^{-1}u_m^{-1}})
\\
\cdots \ol R^{t}_{m-1,m}( {u_{m-1}^{-1}u_m^{-1}})S_m(uq^{2-2m}).
\end{split}
\end{equation}
\end{proof}

For any $1\leq i,j\leq N$ we define the  elements $s^\sharp_{ij}(u)$ as follows

\begin{equation}
  \begin{split}
    s^\sharp_{ij}(u)=\left \{
    \begin{aligned}
      \frac{u^{-1}-u }{qu^{-1}-q^{-1}u}s_{ij}(u)+ \frac{(q-q^{-1})u^{-1}}{qu^{-1}-q^{-1}u }s_{ji}(u ),   i<j,\\
      \frac{u^{-1}-u}{qu^{-1}-q^{-1}u}s_{ij}(u)+\frac{(q-q^{-1})u}{qu^{-1}-q^{-1}u}s_{ji}(u),  i>j\\
       s_{ii}(u) ,   i=j.\\
    \end{aligned}
  \right.
  \end{split}
  \end{equation}
Note that in the orthogonal case, $s^\sharp_{ij}(u)=\ol s_{ji}(u^{-1})$.

\begin{proposition}\label{expansion auxiliary det}
Suppose $i_1<i_2<\dots<i_{m-1}$, $j_2<\dots,j_{m-1}$,
$j_1\in \{i_1,\dots, i_m\}$ and $c\notin \{j_2,\dots,j_{m-1}\}$.
Then
\begin{equation}
\begin{split}
\check s^{i_1\cdots i_m}_{j_1\cdots j_{m-1},c}(u)&=0, \qquad \text{ if } c\notin \{i_1,\dots, i_m\},  \\
\check s^{i_1\cdots i_m}_{j_1\cdots j_{m-1},c}(u)&=
\sum_{r=1}^{m-1}
(-q)^{1-r}s^\sharp_{i_r,j_1}(u)
 s^{i_1,\dots,\hat{i_r},\dots,i_{m-1}}_{j_2,\dots,j_{m-1}}(uq^{-2}),
\end{split}
\end{equation}
if $c=i_m$.

\end{proposition}
\begin{proof}
For $c\notin \{j_2,\dots,j_{m-1}\}$, we have that
\begin{equation}
\ol R^{t}_{2m}\cdots \ol R^{t}_{m-1,m}e_{j_2}\otimes \dots \otimes e_{j_{m-1}}\otimes e_c
=e_{j_2}\otimes \dots \otimes e_{j_{m-1}}\otimes e_c.
\end{equation}

We compute that
\begin{equation}\label{e:expand}
\begin{split}
A_m ^q \langle S_1,\dots,S_{m-1} \rangle
\ol R^{t}_{1m}\cdots \ol R^{t}_{m-1,m}e_{j_1}\otimes \dots \otimes e_{j_{m-1}}\otimes e_c\\
=A_m ^q S_1 \ol R^{t}_{12}\cdots \ol R^{t}_{1,m} \langle S_2,\dots,S_{m-1} \rangle
e_{j_1}\otimes \dots \otimes e_{j_{m-1}}\otimes e_c,\\
\end{split}
\end{equation}
where $S_i=S_i(uq^{2-2i})$.
Let $A_{\{2,\dots,m\}}^q$ be the $q$-antisymmetrizer on the indices $\{2,\dots,m\}$.
Then $A_m^q=A_m^q A_{\{2,\dots,m\}}^q$,   and
\begin{equation}\label{transpose relation}
\begin{aligned}
  & A_{\{2,\dots,m\}}^q \ol R^{t}_{12} \cdots \ol R^{t}_{1,m}
= \ol R^{t}_{1m}\cdots \ol R^{t}_{12} A_{\{2,\dots,m\}}^q \\
&= A_{\{2,\dots,m\}}^q \ol R^{t}_{12}\cdots \ol R^{t}_{1,m} A_{\{2,\dots,m\}}^q
= A_{\{2,\dots,m\}}^q\ol R^{t}_{1m}\cdots\ol R^{t}_{12} A_{\{2,\dots,m\}}^q
\end{aligned}
\end{equation}
which follows from the  Yang-Baxter equation \eqref{YBE}.

Then we can rewrite \eqref{e:expand} as
\begin{equation}\label{expansion}
\begin{split}
&A_m^q S_1 \ol R^{t}_{12}\cdots \ol R^{t}_{1,m}
A_{\{2,\dots,m\}}^q
\langle S_2,\dots,S_{m-1} \rangle
e_{j_1}\otimes \dots \otimes e_{j_{m-1}}\otimes e_c\\
=&A_m ^q S_1 \ol R^{t}_{12}\cdots \ol R^{t}_{1,m}
\sum s^{k_2,\dots,k_{m-1}}_{j_2,\dots,j_{m-1}}(q^{-2}u)
e_{j_1}\otimes e_{k_2}\otimes \dots \otimes e_{k_{m-1}}\otimes e_c
\end{split}
\end{equation}
where the sum runs through ${k_2<\dots<k_{m-1}}$.
Now let's compute the coefficient of
$e_{i_1}\otimes \dots\otimes e_{i_m}$ in
\begin{equation} \label{coeff}
  \begin{split}
A_m ^q S_1 \ol R^{t}_{12}\cdots \ol R^{t}_{1,m}
  e_{j_1}\otimes e_{k_2}\otimes \dots \otimes e_{k_{m-1}}\otimes e_c.
  \end{split}
  \end{equation}
This is done in several cases: 

(i) If $c\notin \{i_1,\dots, i_m\}$, then $c\neq j_1$ and
\begin{equation}
\ol R_{1m}^t
e_{j_1}\otimes e_{k_2}\otimes \dots \otimes e_{k_{m-1}}\otimes e_c
=
e_{j_1}\otimes e_{k_2}\otimes \dots \otimes e_{k_{m-1}}\otimes e_c.
\end{equation}
So the basis vectors $e_{r_1}\otimes \dots\otimes e_{r_m}$ in the expansion of
\eqref{expansion}
only contain those with $c\in \{r_1,\dots, r_m\}$, thus
$\check s^{i_1\cdots i_m}_{j_1\cdots j_{m-1},c}(u)=0$.

(ii) If $c=j_1=i_m$, then
\begin{equation}
A_{\{2,\dots,m\}}^q
e_{j_1}\otimes e_{k_2}\otimes \dots \otimes e_{k_{m-1}}\otimes e_c=0
\end{equation}
whenever $k_r=j_1$ for some $2\leq r\leq m-1$.
In order to have $e_{i_1}\otimes e_{i_2}\dots\otimes e_{i_m}$,
$\{k_2,\cdots, k_{m-1}\}$ must be $\{i_1,i_2,\cdots,\hat {i_r},\cdots, i_{m-1}\}$
for some $1\leq r\leq m-1$ and $j_1\notin \{i_1,i_2,\cdots,\hat {i_r},\cdots, i_{m-1}\}$.

Assume that $i_p<j_1<i_{p+1}$.
For $r<p$,
the coefficient of
$
e_{i_1} \otimes \dots  \otimes e_{i_m}
$
in \eqref{coeff}
is

\begin{equation}
  \begin{split}
  &(-q)^{1-r}\frac{q^{-1}x-q}{x-1}s_{i_r j_1}(u)+(-q)^{-r}\frac{(q^{-1}-q)x}{x-1}s_{j_1i_r}(u),
   \end{split}
  \end{equation}
where $x=q^2u^{-2}$.
This element can be written as   $ (-q)^{1-r}s^{\sharp}_{i_r,j_1}(u)$.
  For $r>p$,
  the coefficient of
  $
  e_{i_1} \otimes \dots  \otimes e_{i_m}
  $
  in
  \eqref{coeff}
is
\begin{equation}
  \begin{split}
  &(-q)^{1-r}\frac{q^{-1}x-q}{x-1}s_{i_r j_1}(u)+(-q)^{2-r}\frac{(q^{-1}-q)}{x-1}s_{j_1i_r}(u),
   \end{split}
  \end{equation}
which is equal to $ (-q)^{1-r}s^{\sharp}_{i_r,j_1}(u)$.

(iii) Suppose $c=i_m$ and $j_1= i_{p}$ for some $1\leq p\leq m-1$.

If $j_1\notin\{k_2,\dots,k_{m-1}\}$, then
 \begin{equation}
\begin{split}
A_m ^q S_1 R^{t}_{12}\cdots R^{t}_{1,m}
e_{j_1}\otimes e_{k_{2}}\otimes \dots \otimes e_{k_{m-1}}\otimes e_c
\\
=A_m ^q \sum_{k_1=1}^Ns_{k_1j_1}(u)
e_{k_1}\otimes e_{k_2}\otimes \dots \otimes e_{k_{m-1}}\otimes e_c.\\
\end{split}
\end{equation}

If $j_1=k_s$ for $2\leq s\leq m-1$, then
\begin{equation}
\begin{split}
&A_m ^q S_1 R^{t}_{12}\cdots R^{t}_{1,m}
e_{j_1}\otimes e_{k_2}\otimes \dots \otimes e_{k_{m-1}}\otimes e_c\\
=&(-q)^{s-2}A_m^q S_1 R^{t}_{12}\cdots R^{t}_{1,m}
e_{j_1}\otimes e_{k_s}\otimes e_{k_2} \dots \hat{e_{k_s}}\dots e_{k_{m-1}}\otimes e_c.\\
\end{split}
\end{equation}
In both cases, in order to have $e_{i_1}\otimes e_{i_2}\dots\otimes e_{i_m}$,
$\{k_2,\cdots, k_{m-1}\}$ must be $\{i_1,i_2,\cdots,\hat {i_r},\cdots, i_{m-1}\}$
for some $1\leq r\leq m-1$.

The coefficient of $e_{i_1}\otimes e_{i_2}\dots \otimes e_{i_m}$ in
\eqref{coeff} is
\begin{equation}
    \begin{cases}
      (-q)^{1-r}\frac{q^{-1}x-q}{x-1}s_{i_r j_1}(u)+(-q)^{-r}\frac{(q^{-1}-q)x}{x-1}s_{j_1i_r}(u), & r<p,\\
            (-q)^{1-r}\frac{q^{-1}x-q}{x-1}s_{i_r j_1}(u)+(-q)^{2-r}\frac{(q^{-1}-q)}{x-1}s_{j_1i_r}(u), & r>p\\
            (-q)^{1-r}s_{i_r,j_1}(u) , &  r=p.
    \end{cases}
  \end{equation}
where $x=q^2u^{-2}$. All of these are equal to $ (-q)^{1-r}s^{\sharp}_{i_r,j_1}(u)$.

Therefore, in cases (ii) and (iii)
\begin{equation}
\begin{split}
\check s^{i_1\cdots i_m}_{j_1\cdots j_{m-1},c}(u)&=
\sum_{r=1}^{m-1}
(-q)^{1-r}  s^{\sharp}_{i_r,j_1}(u)
 s^{i_1,\dots,\hat{i_r},\dots,i_{m-1}}_{j_2,\dots,j_{m-1}}(uq^{-2}).
\end{split}
\end{equation}
This completes the proof.

\end{proof}

In order to produce an explicit formula for the Sklyanin determinant in the orthogonal case we introduce a map
\begin{center}
$\pi_{N}:S_{N}\rightarrow S_{N},\ p\mapsto p'$
\end{center}
which was used in the  formula for the Sklyanin determinant for the twisted Yangians \cite{M, MRS}. The
map $\pi_N$ is defined inductively as follows. Given a set of positive integers $\omega_{1}<\cdots<\omega_{N}$, and consider the action of
 $S_{N}$ on these indices. If $N=2$ we define $\pi_{2}$ as the identity map of $S_{2}\rightarrow S_{2}$.
 For $N>2$ define a map from the set of ordered pairs $(\omega_{k},\ \omega_{l})$ with $k\neq l$ into themselves by the rule
\begin{equation}\label{permutation map}
\begin{split}
&(\omega_{k},\ \omega_{l})\mapsto(\omega_{l},\ \omega_{k})\ ,\ k,\ l<N,
\\
&(\omega_{k},\ \omega_{N})\mapsto(\omega_{N-1},\ \omega_{k})\ ,\ k<N-1,
\\
&(\omega_{N},\ \omega_{k})\mapsto(\omega_{k},\ \omega_{N-1})\ ,\ k<N-1,
\\
&(\omega_{N-1},\ \omega_{N})\mapsto(\omega_{N-1},\ \omega_{N-2})\ ,
\\
&(\omega_{N},\ \omega_{N-1})\mapsto(\omega_{N-1},\ \omega_{N-2})\ .
\\
\end{split}
\end{equation}
Let $p=(p_{1},\ \ldots,p_{N})$ be a permutation of the indices $\omega_{1}, \ldots, \omega_{N}$. Its image under the map $\pi_{N}$ is the permutation $p'=(p_{1}',\ \ldots,p_{N-1}',\ \omega_{N})$ , where the pair $(p_{1}',p_{N-1}')$ is the image of the ordered pair $(p_{1},p_{N})$ under the map \eqref{permutation map}. Then the pair $(p_{2}',p_{N-2}')$ is found as the image of $(p_{2},p_{N-1})$ under the map \eqref{permutation map} which is defined on the set of ordered pairs of elements obtained from $(\omega_{1},\ \ldots,\ \omega_{N})$ by deleting $p_{1}$ and $p_{N}$.
The procedure is completed in the same manner by determining consecutively
the pairs  $(p_{i}',p_{N-i}')$.

In the following theorem, we give explicit formula for Sklyanin determinants. Note that the explicit formula for Sklyanin determinants in orthogonal case was given in \cite{MRS}.
\begin{theorem}\label{sdet formula}
The Sklyanin determinant $\sdet_{q}(S(u))$ can be written explicitly as
\begin{equation}
  \begin{aligned}
  \sdet_{q}(S(u)) &=  \sum_{p\in S_N}(-q)^{l(p')-l(p)}
  s^\sharp _{p_1p_1'}(u)\cdots  s^\sharp_{p_np_n'}(q^{2-2n}u)\\
  &\times
s_{p_{n+1}p_{n+1}'}(q^{-2n}u)\cdots s_{p_Np_N'}(q^{2-2N}u).
\end{aligned}
\end{equation}
\end{theorem}
\begin{proof} For $i_1<i_2\cdots <i_m$, we can write
 \begin{equation}
  \begin{split}
s^{i_1,\cdots,i_m}_{i_1,\cdots,i_{m-1},j_m}(u)
  =\sum_{k=1}^{m}\check s^{i_1,\cdots,i_m}_{i_1,\cdots,i_{m-1},i_k}(u)
  s_{i_k,j_m}(q^{2-2m}u).\\
  \end{split}
  \end{equation}
  It follows from Proposition \ref{expansion auxiliary det} that
  \begin{align*}
 s&^{i_1,\cdots,i_m}_{i_1,\cdots,i_{m-1},j_m}
 =s^\sharp_{i_{m-1},i_{m-1}}(u )
  s^{i_1, \cdots,i_{m-2}}_{i_1,\cdots, i_{m-2} }(q^{-2}u)
  s_{i_m, j_m }(q^{2-2m}u)  \\
  &+(-q)^{2m-2l+3}\sum_{l=1}^{m-2} s^\sharp _{i_l,i_{m-1}}(u)
 s^{i_1,\cdots,\hat{i_l},\cdots,i_{m-1}}_{i_1,\cdots,\hat{i_l},\cdots,i_{m-2},i_l}
 (q^{-2}u)
 s_{i_m, j_m }(q^{2-2m}u) \\
 &+(-q)^{2k-2m+1}\sum_{k=1}^{m-1} s^\sharp _{i_m,i_k} (u)
 s^{i_1,\cdots,\hat{i_k},\cdots,i_{m-1}}
  _{i_1,\cdots,\hat{i_k},\cdots,i_{m-1}}(q^{-2}u)
  s_{i_k, j_m }(q^{2-2m}u)
 \\
  &+(-q)^{2m-2k-2l}\sum_{k=1}^{m-1}\sum_{l=1}^{k-1}
  s^\sharp_{i_l,i_k}(u)
 s^{i_1,\cdots,\hat{i_l},\cdots,\hat{i_k},\cdots,i_m}
  _{i_1,\cdots,\hat{i_l},\cdots,\hat{i_k},\cdots,i_{m-1},i_l}(q^{-2}u)
 s_{i_k, j_m }(q^{2-2m}u) \\
  \\
  &+(-q)^{2m-2k-2l+2}\sum_{k=1}^{m-1}\sum_{l=k+1}^{m-1}
  s^\sharp_{i_l,i_k}(u)
  s^{i_1,\cdots,\hat{i_k},\cdots,\hat{i_l},\cdots,i_m}
  _{i_1,\cdots,\hat{i_k},\cdots,\hat{i_l},\cdots,i_{m-1},i_l}(q^{-2}u)
 s_{i_k, j_m }(q^{2-2m}u). \\
   \end{align*}
  Starting with $s^{1,\cdots,N}_{1,\cdots,N}(u)$, we apply the recurrence relation repeatedly to write
  the Sklyanin determinant $\sdet_{q}(X)$ in terms of the generator $s_{ij}$:
  \begin{equation}
    \begin{split}
    \sdet_{q}(S(u))=  \sum_{p\in S_N}(-q)^{l(p')-l(p)}
    s^\sharp_{p_1p_1'}(u)\cdots s^\sharp_{p_np_n'}(q^{2-2n}u)\\
    \times
  s_{p_{n+1}p_{n+1}'}(q^{-2n}u)\cdots x_{p_Np_N'}(q^{2-2N}u).
  \end{split}
  \end{equation}
  \end{proof}

\subsection{Minor identities for Sklyanin determinants}

We define the Sklyanin comatrix $\wh S(u)$ by
\begin{equation}
\wh S(u) S(q^{2-2N}u)=\sdet_{q}(S(u))I.
\end{equation}

\begin{proposition}\label{Skl comatrix}
The matrix elements  $\hat s_{ij}(u)$ are given by
\begin{equation}
\hat s_{ij}(u)=(-q)^{N-i}\check s^{1,\cdots,N}_{1,\cdots \hat i,\cdots ,N,j}(u)
\end{equation}
for $i\neq j$ and
\begin{equation}
\hat s_{ii}(u)=s^{1,\cdots \hat i,\cdots ,N}_{1,\cdots \hat i,\cdots ,N}(u).
\end{equation}
\end{proposition}

\begin{proof}
Multiplying $S_N(q^{2-2N}u)^{-1}$ from the right of the formulas
\begin{equation}
A_N^q \langle S_1,\dots,S_N \rangle
=A_N^q \sdet_q(S(u)),
\end{equation}
we get that
\begin{equation}
A_N^q \langle S_1,\dots,S_{N-1} \rangle
\ol R^{t}_{1N}\cdots \ol R^{t}_{N-1,N}=A_N^q \hat S_{N}(u).
\end{equation}
Applying both sides to the vector
\begin{equation}
v_{ij}=e_1\otimes\cdots\hat{e_{i}}\otimes e_{N}\otimes e_j
\end{equation}
and comparing the coefficients of
$e_1\otimes\cdots \otimes e_{N} $ we get the first formula.
Using $\ol R^t_{kN}v_{ii}=v_{ii}$ for $1\leq k\leq N-1$, applying the operators to the vector
$v_{ii}$ we obtain the second formula.
\end{proof}

Let $C=\diag((-q)^{\frac{N-1}{2}}, (-q)^{\frac{N-3}{2}},\ldots, (-q)^{-\frac{N-1}{2}})$ be the $N\times N$ diagonal matrix.
We have that
\begin{equation}
R(x)C_1C_2 = C_1C_2R(x),
\end{equation}
and the crossing symmetry relations can be written as
\begin{equation}
  \begin{split}
  R_{12}^{-1}(x)^{t_2}C_1C_2R_{12}^{t_2}(xq^{2N})=C_1C_2,\\
  R_{12}^{t_1}(xq^{2N}) C_1C_2 R_{12}^{-1}(x)^{t_1}=C_1C_2.
 \end{split}
  \end{equation}
We see that
\begin{equation}
  \begin{split}
(\ol R(x)^t)^{-1}= C_1C_2 \ol R'(q^{-2N}x)^tC_1^{-1}C_2^{-1},
 \end{split}
  \end{equation}
where $\ol R'(x)$ is obtained from $ \ol  R(x)$ by replacing $q$ with  $q^{-1}$.
Taking the inverse of both sides of the reflection equation
\begin{equation}
  \begin{split}
R(u/v)S_{1}(u)R^{t}(1/uv)S_{2}(v)=S_{2}(v)R^{t}(1/uv)S_{1}(u)R(u/v),
 \end{split}
  \end{equation}
subsequently 
\begin{equation}
  \begin{split}
  R'(u/v) S_{1}(u)^{-1}C_1C_2\left(R '(1/q^{2N}uv)\right)^{t}C_1^{-1}C_2^{-1}S_{2}(v)^{-1}\\
  =S_{2}(v)^{-1}C_1C_2\left(R '(1/q^{2N}uv)\right)^{t}C_1^{-1}C_2^{-1}S_{1}(u)^{-1}R '(u/v). \end{split}
  \end{equation}
Multiplying $C_1^{-1}C_2^{-1}$ from the left and $C_1C_2$ from the right of both sides,
\begin{equation}
  \begin{split}
  R '(u/v) C_1^{-1} S_{1}(u)^{-1}C_1 \left(R ' (1/q^{2N}uv)\right)^{t} C_2^{-1}S_{2}(v)^{-1} C_2\\
  = C_2^{-1}S_{2}(v)^{-1} C_2\left(R ' (1/q^{2N}uv)\right)^{t}C_1^{-1} S_{1}(u)^{-1}C_1  R' (u/v).
   \end{split}
  \end{equation}
Denote
$C^{-1} S (q^{-N}u)^{-1} C$ by $X(u)$,
then $X(u)$ satisfies the $q^{-1}$-reflection relation, i.e.
  \begin{equation}\label{inverse reflection rel}
  \begin{split}
  R '(u/v) X_1(u)  \left(R '(1/ uv)\right)^{t}  X_2(u)\\
  = X_2(u) \left(R ' (1/ uv)\right)^{t}  X_1(u)     R '(u/v).
 \end{split}
  \end{equation}

We write $S(u)$ as
\begin{equation}
\begin{split}
s_{ij}(u)=\displaystyle \sum_{r=0}^{\infty}S^{(r)}u^{-r}.
\end{split}
\end{equation}

In the orthogonal case, the matrix $S^{(0)}$ is a lower unitriangular matrix, it is clearly invertible
 and the inverse of $S^{(0)}$  is lower unitriangular.
In the symplectic case, $S^{(0)}$  is invertible and
\begin{equation}
\begin{pmatrix}
({S^{(0)}})^{-1}_{ii} & {({S^{(0)}})^{-1}}_{ii'} \\
 {({S^{(0)}})^{-1}}_{i'i} & {({S^{(0)}})^{-1}}_{i'i'}
\end{pmatrix}
=
q^{-3}\begin{pmatrix}
  s_{i'i'} &-q^{2} s_{ii'} \\
  (q-q^{-1})s_{ii'}-s_{i'i} & s_{ii}
\end{pmatrix}
\end{equation}
therefore
\begin{equation}
\begin{aligned}
 &{(S^{(0)})^{-1}}_{i'i'}{(S^{(0)})^{-1}}_{ii}-q^{-2}{(S^{(0)})^{-1}}_{i'i}{(S^{(0)})^{-1}}_{ii'}\\
 &=q^{-6}s_{ii}s_{i'i'}+q^{-6}((q-q^{-1})s_{ii'}-s_{i'i})  s_{ii'}=q^{-3}.
 \end{aligned}
\end{equation}
Also we have that 
\begin{equation}
{(S^{(0)})^{-1}}_{ij}=0 \text{ for } i<j \text{ with } j\neq i'.
\end{equation}

Combining with \eqref{inverse reflection rel}, we have the following proposition.
\begin{proposition}
The mapping $S(u)\mapsto X(u)$, $q\mapsto q^{-1}$, defines an automorphism of $\mathrm Y_q^{\tw}(\g_N)$.
\end{proposition}

The following result is the Sklyanin determinant analogue of Jacobi's theorem.

\begin{theorem}\label{skly jacobi thm}
Let $I=\{i_1<i_2<\cdots<i_k\}$ be a subset of $[1,N]$ and
$I^{c}=\{i_{k+1}<\cdots<i_{N}\}$ the complement of $I$. Then
\begin{equation}\label{e:sklJacobi}
\begin{split}
\sdet_q S_I(u)=
\sdet_q S(u)  \sdet_{q^{-1}}X_{I^c} (q^{2-N}u),
\end{split}
\end{equation}
where
$X(u)=C^{-1} S (q^{-N}u)^{-1} C $.
\end{theorem}

\begin{proof}
The $q$-antisymmetrizer satisfies the relation
\begin{equation}
C_1\dots C_N A_{N}^{q} C_1^{-1}\dots C_{N}^{-1}=A_N^{q}.
\end{equation}

By the relation
\begin{equation}
A_{N}^q \langle S_1,\dots,S_{N} \rangle =\sdet_q S(u) A^q_N
\end{equation}
and the definition of $\langle S_1, \cdots, S_N\rangle$ we have that
 \begin{equation}
 \begin{split}
&A^q_{N}\langle S_1,\dots,S_{k} \rangle
\overrightarrow{\prod}_{1\leq i\leq k<j\leq N}\ol R_{ij}^t\\
&=
\sdet_q(X) A_{N}^q
S^{-1}_{N}(\ol R_{N-1,N}^t)^{-1}S_{N-1}\cdots  (\ol R_{k+1,k+2}^t)^{-1}S_{k+1}^{-1}
\end{split}
\end{equation}
where $u_i=uq^{-2i+2}$, $S_i=S(u_i)$ and $\ol R_{ij}^t=\ol R_{ij}^t(1/u_iu_j)$.
\begin{equation}\label{e:4.67}
 \begin{aligned}
&A_{N}^q \langle S_1,\dots,S_{k} \rangle
\overrightarrow{\prod}_{1\leq i\leq k<j\leq N}\ol R_{ij}^t\\
&=\sdet_q(S(u)) A_{N}^q
S^{-1}_{N}(\ol R_{N-1,N}^t)^{-1}S_{N-1}\cdots  (\ol R_{k+1,k+2}^t)^{-1}S_{k+1}^{-1}\\
&= \sdet_q S(u) A_{N}^q C_{k+1}\cdots C_{N}
X_{N}(q^Nu_N)(\ol R_{N-1,N}'(1/q^{2N} u_{N-1}u_N))^{t}
\\
&\quad X_{N-1}(q^Nu_{N-1})\cdots  (\ol R_{k+1,k+2} ' (1/q^{2N} u_{k+1}u_{k+2} ))^tX_{k+1}(q^Nu_{k+1})\\
& \qquad\quad C_{k+1}^{-1}\cdots C_{N}^{-1}.
\end{aligned}
\end{equation}

Applying both sides to the vector $v=e_{i_1}\otimes e_{i_2}\otimes\dots \otimes e_{i_{k}}\otimes e_{i_N}\otimes \dots e_{i_{k+1}}$, we see that $v$'s
coefficient in $A_{N}^q \langle S_1,\dots,S_{k} \rangle
\overrightarrow{\prod}_{1\leq i\leq k<j\leq N}\ol R_{ij}^t v$ (the LHS) is $\sdet S_I(u)$ due to the fact that
\begin{equation}
{\prod}_{1\leq i\leq k<j\leq N}\ol R_{ij}(1/u_iu_j)^tv=v.
\end{equation}
The coefficient of $v$ in the RHS of \eqref{e:4.67}
is
$
\sdet_{q^{-1}}X_{I^c} (q^{2-N}u),
$
 which has proved \eqref{e:sklJacobi}.
\end{proof}

The following is the special case of Theorem \ref{skly jacobi thm}
for $I=\{1,2,\cdots,N\}$.
\begin{corollary}\label{inverse sdet} In the algebra $\mathrm{Y}_{q}^{\tw}(\g_N)$ we have that
  \begin{equation}
    \begin{split}
     \sdet_q S(u)  \sdet_{q^{-1}}X(q^{2-N}u)=1.
    \end{split}
    \end{equation}

\end{corollary}

\begin{theorem}
\label{jacobi comatrix}
 For any $1\leq a,b\leq N$, one has that
  \begin{equation}
s^{a,N+1,\cdots,N+M}_{b,N+1,\cdots,N+M}(u) =(-q)^{b-N}\sdet_q(S(u)) \check X_{1,\cdots,\hat{a},\cdots,N,b}^{1,\cdots,N}(q^{2-N-M}u),
  \end{equation}
  where
$X(u)=C ^{-1} S (q^{-N-M}u)^{-1} C $.
\end{theorem}

  \begin{proof}
    The proof is similar to Jacobi's theorem.
    Using the relation
  \begin{equation}
  \begin{split}
    &A_{N+M} \langle S_1,\dots,S_{N+1} \rangle
    \overrightarrow{\prod}_{1\leq i\leq N,N+2\leq j\leq N}\ol R_{ij}^t\\
  =&
  \sdet_q(S(u)) A_{N+M} C_{N+1}\dots C_{N+M}
  X_{N+M}(q^{N+M}u_{N+M})(\ol R_{N-1,N}')^{t}\cdots\\
  &\cdots X_{N+2}(q^{N+M}u_{N+2})\\
  &\cdot (\ol R_{N+1,N}')^{t}\dots (\ol R_{N+1,N+2}')^{t}  C_{N+1}^{{-1}}\dots C_{N+M}^{-1}
  \end{split}
  \end{equation}
  Applying both sides to the vector
  $e_{N+1}\otimes \dots e_{N+M}\otimes e_{b}\otimes e_N \otimes \dots \otimes \widehat {e_{a}}\dots \otimes e_{1}$
  and comparing the
  coefficient of $e_{1}\otimes e_{2}\otimes  \dots \otimes e_{N+M}$ we obtain that
  \begin{equation}
  s^{a,N+1,\cdots,N+M}_{b,N+1,\cdots,N+M} (u)=(-q)^{b-N}\sdet_q(S(u)) \check X_{1,\cdots,\hat{a},\cdots,N,b}^{1,\cdots,N}(q^{2-N-M}u)
  \end{equation}
  \end{proof}

The following is an analogue of Schur's complement theorem.

\begin{theorem}
We write
\begin{equation}
  S(u)=
  \begin{pmatrix}
      S_{11}(u) &S_{12}(u)\\
      S_{21}(u) &S_{22}(u)
    \end{pmatrix},
\end{equation}
where $S_{11}$ is $N\times N$ matrix and $S_{22}$ is $M\times M$ matrix. In symplectic case, $N$ and $M$ are even.
Then
\begin{equation*}
  \begin{split}
    &\sdet_q(S(u))\\
    =&\sdet_q(S_{11}(u))  \sdet_q(S_{22}(q^{-2N}u)-S_{21}(q^{-2N}u)S_{11}(q^{-2N}u)^{-1}S_{12}(q^{-2N}u))\\
    =&\sdet_q(S_{22}(u))  \sdet_q(S_{22}(q^{-2M}u)-S_{21}(q^{-2M}u)S_{11}(q^{-2M}u)^{-1}S_{12}(q^{-2M}u)).
  \end{split}
  \end{equation*}
\end{theorem}

  \begin{proof}
We subdivide $C$ and $X(u)=C ^{-1} S (q^{-N-M}u)^{-1} C $  into blocks
\begin{equation}
  C=
  \begin{pmatrix}
      C_{11} &0\\
      0 &C_{22}
    \end{pmatrix},   \qquad
  X(u)=
  \begin{pmatrix}
      X_{11}(u)&X_{12}(u)\\
      X_{21}(u) &X_{22}(u)
    \end{pmatrix},
 \end{equation}
then
\begin{align} \notag
    X_{11}(u)&={C_{11}}^{-1} \left(S_{11}(q^{-M-N}u)-S_{12}(q^{-M-N}u)S_{22}(q^{-M-N}u)^{-1}\right.\\
                   &\qquad\qquad\qquad \left. S_{21}(q^{-M-N}u)\right) ^{-1}  C_{11},\\ \notag
    X_{22}(u)&= {C_{22}}^{-1}\left(S_{22}(q^{-M-N}u)-S_{21}(q^{-M-N}u)S_{11}(q^{-M-N}u)^{-1}\right.\\
                   &\qquad\qquad\qquad \left. S_{12}(q^{-M-N}u)\right) ^{-1}C_{22}.
\end{align}

It follows form Theorem \ref{skly jacobi thm} that
\begin{equation}
  \begin{split}
    \sdet_q(S_{11}(u))=\sdet_q(S(u))\sdet_{q^{-1}}(X_{22}(q^{2-N-M}u).
  \end{split}
\end{equation}

Since
$X_{22}(u)$ satisfies the $q^{-1}$ reflection relation,
\begin{equation}
  \begin{aligned}
    &{C_{22}}X_{22}(q^M u)^{-1}  {C_{22}}^{-1}\\
    &=S_{22}(q^{-N}u)-S_{21}(q^{-N}u)S_{11}(q^{-N}u)^{-1}S_{12}(q^{-N}u)
  \end{aligned}
\end{equation} satisfies the $q$ reflection relation.
By Corollary \ref{inverse sdet},
\begin{equation}
  \begin{aligned}
&\sdet_q(S_{22}(q^{M-N-2}u)-S_{21}(q^{M-N-2}u)S_{11}(q^{M-N-2}u)^{-1}\\
& \qquad S_{12}(q^{M-N-2}u))\sdet_{q^{-1}} X_{22}(u) =1.
\end{aligned}
\end{equation}
Therefore,
\begin{equation}
\begin{aligned}
    &\sdet_q(S(u))\\
    &=\sdet_q(S_{11}(u))  \sdet_q(S_{22}(q^{-2N}u)-S_{21}(q^{-2N}u)S_{11}(q^{-2N}u)^{-1}\\
    &\qquad\qquad S_{12}(q^{-2N}u)).\\
  \end{aligned}
\end{equation}

The second equation can be proved similarly.
  \end{proof}

Using Jacobi's theorem we obtain the
following analogue of Cayley's complementary identity for the Sklyanin determinant.
\begin{theorem}\label{skly cayley thm }
Suppose a minor identity for the Sklyanin determinant is given:
\begin{equation}\label{skly identity}
\sum_{i=1}^{k}b_i \prod_{j=1}^{m_i}
\sdet_{q}  S_{I_{ij}}(u)=0,
\end{equation}
where $I_{ij}'s$ are subsets of $[1,N]$  and $b_i\in \mathbb C(q)$.
Then the following identity holds
\begin{equation}
\sum_{i=1}^{k} b_i'\prod_{j=1}^{m_i}
\sdet_q S(u)^{-1}\sdet_q S_{I_{ij}^c}(u)=0,
\end{equation}
where $b_i'$ is obtained from $b_i$ by replacing $q$ by $q^{-1}$.
\end{theorem}

\begin{proof}
The matrix
$X(u)=C ^{-1} S (q^{-N}u)^{-1} C $ satisfies the $q^{-1}$ reflection relations. Applying the minor identity to $X(u)$ we get that
\begin{equation}\label{cayley eq1}
\sum_{i=1}^{k}b_i' \prod_{j=1}^{m_i}
\sdet_{q^{-1}}(X_{I_{ij}}(u))=0.
\end{equation}
It follows from Theorem \ref{skly jacobi thm} that
\begin{equation}
  \sum_{i=1}^{k} b_i'\prod_{j=1}^{m_i}
  \sdet_q S(q^{N-2}u)^{-1}\sdet_q S_{I_{ij}^c}(q^{N-2}u)=0,
  \end{equation}
The proof is completed by replacing $u$ with $q^{2-N}u$.

\end{proof}

The following theorem is an analogue of Muir's law for the Sklyanin determinant.
\begin{theorem}\label{skly muir law}
Suppose there is a minor Sklyanin determinant identity
\begin{equation}
\sum_{i=1}^{k}b_i \prod_{j=1}^{m_i}
\sdet_{q} S_{I_{ij}}(u) =0,
\end{equation}
where $I_{ij}'s$ are subsets of $I=\{1,2,\dots,N\}$ and $b_i\in \mathbb C(q)$.
Let $J$ be the set $\{N,\dots, N+M\}$. Then the following identity holds
\begin{equation}
\sum_{i=1}^{k} b_i\prod_{j=1}^{m_i}
\sdet_{q} (S_{ J})^{-1}
\sdet_{q} (S_{I_{ij}\cup J})=0.
\end{equation}
\end{theorem}
\begin{proof}
Applying Cayley's complementary identity respect to the set $I$, we get that
\begin{equation}\label{obtain by jacobi thm}
\sum_{i=1}^{k} b_i'\prod_{j=1}^{m_i}
\sdet_q S_{I}(u)^{-1}
\sdet_{q} S_{I\setminus I_{ij}}(u)=0,
\end{equation}
Applying Cayley's
complementary identity respect to the set $I\cup J$,
we obtain that
\begin{equation}
\sum_{i=1}^{k} b_i\prod_{j=1}^{m_i}
\sdet_q S_{J}(u)^{-1}
\sdet_{q} S_{I_{ij}\cup J}(u)
=0.
\end{equation}

\end{proof}

In the following we give an analogue of Muir's identities for the Sklyanin determinant.


\begin{lemma}\label{A H relation} One has that
\begin{equation}\label{H relation}
    \begin{aligned}
A_{\{2,\dots,m\}}^q R^{t}_{12}(u)&\cdots R^{t}_{1,m}(uq^{2m-2})\\
    &=  A_{\{2,\dots,m\}}^q R^{t}_{12}(u)\cdots R^{t}_{1,m}(uq^{2m-2})A_{\{2,\dots,m\}}^q,\\
 R^{t}_{1,m}(u) &\cdots R^{t}_{m-1,m}(uq^{2m-2})H_{m-1}^q\\
    &= H_{m-1}^q R^{t}_{12}(u)\cdots R^{t}_{1,m}(q^{2m-2}u)H_{m-1}^q,\\
    \end{aligned}
    \end{equation}
 where   $A_{\{2,\dots,m\}}^q$ is the  $q$-antisymmetrizer numbered by indices $\{2,\dots,m\}$.
  \end{lemma}

\begin{proof}
 Taking transposition with respect to the first factor to Yang-Baxter equation we have
  \begin{equation}
    \begin{split}
    R_{23}(v,w)R_{12}^t(u,v)R_{13}^t(u,w)=R_{13}^t(u,w)R_{12}^t(u,v)R_{23}(v,w).
    \end{split}
  \end{equation}
  Taking $v=1,w=q^{-2}$ we have that
  \begin{equation}
    \begin{split}
  A_{\{2,3\}}^q R_{12}^t(u)R_{13}^t(uq^2)=R_{13}^t(uq^2)R_{12}^t(u)A_{\{2,3\}}^q.
    \end{split}
  \end{equation}

Since $(A_2^q )^2=A_2 ^q $, we have
  \begin{equation}
    \begin{aligned}
   A_{\{2,3\}}^q R_{12}^t(u)R_{13}^t(uq^2)  A_{\{2,3\}}^q &=   A_{\{2,3\}}^q R_{12}^t(u)R_{13}^t(uq^2),\\
     A_{\{2,3\}}^q R_{12}^t(u)R_{13}^t(uq^2)P_{23}^q&=-   A_{\{2,3\}}^q R_{12}^t(u)R_{13}^t(uq^2).
    \end{aligned}
  \end{equation}

  It implies that
  \begin{equation}\label{H relation}
    \begin{aligned}
&A_{\{2,\dots,m\}}^q R^{t}_{12}(u)\cdots R^{t}_{1,m}(uq^{2m-2})
P_{i,i+1}\\
&=-A_{\{2,\dots,m\}}^q R^{t}_{12}(u)\cdots R^{t}_{1,m}(uq^{2m-2})
    \end{aligned}
    \end{equation}
for any $1\leq i\leq m-1$. Using the formula for $A_m^q$ we obtain the first equation.

Taking transpositions with respect to the first and second factors consecutively to the Yang-Baxter equation we have
\begin{equation}
  \begin{split}
R_{13}^t(u,w) R^t_{23}(v,w)  R_{12}^{t_1t_2}(u,v)= R_{12}^{t_1t_2}(u,v) R^t_{23}(v,w)  R_{13}^t(u,w) .
  \end{split}
\end{equation}
Using the relation $R_{12}^{t_1t_2}(u,v)=P_{12}R_{12} (u,v) P_{12}$, we have
\begin{equation}
  \begin{split}
 R_{12} (u,v) R^t_{13}(v,w)  R_{23}^t(u,w)=R_{23}^t(u,w) R^t_{13}(v,w)  R_{12} (u,v).
  \end{split}
\end{equation}

Taking $u=1,v=q^{-2}$ we have that
\begin{equation}
  \begin{split}
 A_2 ^q R^t_{13}(q^{-2}w^{-1})  R_{23}^t(w^{-1})=R_{23}^t(w^{-1}) R^t_{13}(q^{-2}w^{-1})  A_2^q.
  \end{split}
\end{equation}
Replacing $w$ with $q^{-2}u^{-1}$, we get
\begin{equation}
  \begin{split}
 A_2^q R^t_{13}(u)  R_{23}^t(q^2 u)=R_{23}^t(q^2 u) R^t_{13}(u)  A_2^q.
  \end{split}
\end{equation}
Using the relation  $A_2^q+H_2^q=1$  and $(A_2^q)^2=A_2^q$, we have
\begin{equation}
  \begin{split}
 P_{12}^q R^t_{13}(u)  R_{23}^t(q^2 u)H_2^q=R^t_{13}(u)  R_{23}^t(q^2 u)H_2^q.
  \end{split}
\end{equation}
Then the second equation can be obtained by the same arguments of the first equation.
\end{proof}

The following version of MacMahon's theorem holds.
\begin{theorem}
One has that
\begin{align}
\sum_{r=0}^{k}(-1)^r tr_{1,\ldots,k} H_{r}^qA_{\{r+1,\ldots,k\}}^q \langle S_1,\dots,S_k\rangle=0,\\
\sum_{r=0}^{k}(-1)^r tr_{1,\ldots,k}A_r^q H_{\{r+1,\ldots,k\}}^q \langle S_1,\dots,S_k\rangle=0,
\end{align}
where $A_{\{r+1,\ldots,k\}}^q$ and $H_{\{r+1,\ldots,k\}}^q$ denote the antisymmetrizer and symmetrizer over the copies of $End(\mathbb{C}^k)$ labeled by $\{r+1,\ldots,k\}$.
\end{theorem}

\begin{proof}
In the following we show that
\begin{equation}\label{sdet trace replacement}
\begin{aligned}
&tr_{1,\ldots,k}H_r^q A_{\{r+1,\ldots,k\}}^q \langle S_1,\dots,S_k\rangle\\
=&tr_{1,\ldots,k}  \frac{r(k-r+1) }{k}H_r^q A_{\{r,\ldots,k\}}^q \langle S_1,\dots,S_k\rangle\\
 &\quad +tr_{1,\ldots,k} H_r^q A_{\{r,\ldots,k\}}^q \frac{(k-r) (r+1) }{k}H_{r+1}^q A_{\{r+1,\ldots,k\}}^q
\langle S_1,\dots,S_k\rangle.
\end{aligned}
\end{equation}

The element $\langle S_1,\dots,S_k\rangle$ can be written as
\begin{equation}
 \langle S_1,\dots,S_r\rangle
\prod_{1\leq i\leq r \atop r+1\leq j\leq k}
\ol R_{ij}^t
\langle S_{r+1},\dots,S_k\rangle ,
\end{equation}
where the product is taken in the lexicographical order on the pairs $(i,j)$.
It follows from Lemma \ref{A H relation}
that
\begin{align}
A_{\{r+1,\ldots,k\}}^q \prod_{1\leq i\leq r \atop r+1\leq j\leq k}
\ol R_{ij}^t
&=A_{\{r+1,\ldots,k\}}^q \prod_{1\leq i\leq r \atop r+1\leq j\leq k}
\ol R_{ij}^t
A_{\{r+1,\ldots,k\}}^q ,\\
 \prod_{1\leq i\leq r \atop r+1\leq j\leq k}
\ol R_{ij}^tH_r^q
&=H_r^q \prod_{1\leq i\leq r \atop r+1\leq j\leq k}
\ol R_{ij}^t
H_r^q.
\end{align}
Then
\begin{align}
A_{\{r+1,\ldots,k\}}^q \langle S_1,\dots,S_k\rangle
&=A_{\{r+1,\ldots,k\}}^q   \langle S_1,\dots,S_k\rangle
A_{\{r+1,\ldots,k\}}^q ,\\
   H_r^q \langle S_1,\dots,S_k\rangle
&=H_r^q  \langle S_1,\dots,S_k\rangle
H_r^q.
\end{align}
By the relation in the  group algebra of $\mathfrak{S}_{k}$,
\begin{align}
 (k-r+1) A_{\{r,\ldots,k\}}^q
  &=A_{\{r+1,\ldots,k\}}^q -(k-r) A_{\{r+1,\ldots,k\}}^q P^q_{r,r+1} A_{\{r+1,\ldots,k\}}^q ,\\
  (r+1)H_{r+1}^q
  &=H_{r}^q +rH_{r}^q P^q_{r,r+1}H_{r}^q.
  \end{align}

Thus we have that
\begin{equation}
\begin{aligned}
&tr_{1,\ldots,k}  H_r A_{\{r+1,\ldots,k\}}^q P^q_{r,r+1} A_{\{r+1,\ldots,k\}}^q
 \langle S_1,\dots,S_k\rangle\\
 &=tr_{1,\ldots,k}  H_r^q  P^q_{r,r+1} A_{\{r+1,\ldots,k\}}^q
 \langle S_1,\dots, S_k\rangle A_{\{r+1,\ldots,k\}}^q\\
 &=tr_{1,\ldots,k}  H_r^q P^q_{r,r+1} A_{\{r+1,\ldots,k\}}^q
 \langle S_1,\dots,S_k\rangle \\
\end{aligned}
\end{equation}
Similarly,
\begin{equation}
\begin{split}
&tr_{1,\ldots,k}  H_{r}^q P^q_{r,r+1} H_{r}^q  A_{\{r+1,\ldots,k\}}^q
\langle S_1,\dots, S_k\rangle\\
=&tr_{1,\ldots,k}  P^q_{r,r+1} H_{r}^q A_{\{r+1,\ldots,k\}}^q
\langle S_1,\dots, S_k\rangle H_{r}^q\\
=&tr_{1,\ldots,k}  P^q_{r,r+1} A_{\{r+1,\ldots,k\}}^q
\langle S_1,\dots,S_k\rangle H_{r}^q\\
=&tr_{1,\ldots,k}  H_{r}^q P^q_{r,r+1} A_{\{r+1,\ldots,k\}}^q
\langle S_1,\dots,S_k\rangle.
\end{split}
\end{equation}
These imply the equation \eqref{sdet trace replacement}. Therefore we have shown the first equation.
The second equation can be proved by the same arguments.
\end{proof}

\subsection{Sylvester's theorem for the Sklyanin determinant}
The following is an analog of Sylvester's theorem for the Sklyanin determinant.
\begin{theorem}
 Let
 $I=\{1,\cdots,N\}$ ,
 $J=\{N+1,\cdots,N+M\}$, where $N$ and $M$ are positive integers such that $N$ and $M$ are even in the symplectic case.
Then the mapping $s_{ij}(u)\mapsto  s^{i,N+1,\cdots,N+M}_{j,N+1,\cdots,N+M}(q^Mu)$
defines an algebra morphism $\mathrm Y_q^{\tw}(\g_N)\rightarrow \mathrm Y_q^{\tw}(\g_{N+M})$. Denote $\wt{s}_{ij}(u)$ by the image of  $s_{ij}(u)$.
 Then
\begin{equation}
 \sdet_q (\wt S(u))= \Pi_{i=1}^{N-1} \sdet_q(S_{J}(q^{M-2i}u))  \sdet_q(S(q^Mu)).
    \end{equation}
\end{theorem}

\begin{proof}
Let $S(u)$ be  the generator matrix 
for $\mathrm Y_q^{\tw}(\g_{N+M})$, we can write $X(u)=C^{-1}S^{-1}(q^{-N-M}u)C$ as a block matrix
\begin{equation}
  \begin{pmatrix}
      X_{11}(u) &X_{12}(u)\\
      X_{21}(u) &X_{22}(u)
    \end{pmatrix},
\end{equation}
where $X_{11}(u)$ and $X_{22}(u)$ are respectively matrices of size $N\times N$ and $M\times M$.
Then $X_{11}(u)$ satisfies the $q^{-1}$-reflection relation.
The inverse of $X_{11}(q^{2N-2}u)$ is  $(\sdet_{q^{-1}}{X}_{11}(u))^{-1}\hat{X}_{11}(u)$.
Denote  $Z(u)=D X_{11}^{-1}(q^{N}u)D^{-1}$, then
$Z(u)$ satisfies the $q$-reflection relation.

It follows from Proposition \ref{Skl comatrix} that the $(i,j)$-th entry of $\hat{X}_{11}(u)$
is $(-q)^{i-N}\check{X}_{1,\cdots,\hat{i},\cdots,N,j}^{1,\cdots,N}(u)$.
By Theorem \ref{jacobi comatrix} we have that
\begin{equation}
  \begin{split}
 \tilde{s}_{ij}(u)
 &=s^{i,M+1,\cdots,M+N}_{j,M+1,\cdots,M+N}(q^Mu) \\
 &=(-q)^{j-N} \sdet_q(S(q^Mu)) \check X_{1,\cdots,\hat{i},\cdots,N,j}^{1,\cdots,N}(q^{2-N}u)\\
 &=(-q)^{j-i} \sdet_q(S(q^M u)) (\hat{X}_{11}(q^{2-N}u))_{ij}\\
 &= (-q)^{j-i} \sdet_q(S(q^Mu)) \sdet_{q^{-1}}({X}_{11}(q^{2-N}u) ) (X_{11}(q^{N}u)^{-1})_{ij}  \\
 &=\sdet_q(S(q^Mu)) \sdet_{q^{-1}}({X}_{11}(q^{2-N}u) ) z_{ij}(u).\\
\end{split}
\end{equation}
Since $\sdet_q(S(q^Mu))$ and $\sdet_{q^{-1}}({X}_{11}(q^{2-N}u) )$ commute with $z_{ij}(u)$ for any $1\leq i,j\leq N$,
$\tilde S(u)$ satisfies the $q$-reflection relation.
This proves the first statement.

By Jacobi's theorem,
\begin{equation}\label{X Y relation}
  \sdet_q(S(q^Mu)) \sdet_{q^{-1}}({X}_{11}(q^{2-N}u) ) =\sdet_q(S_J(q^M u)) ,
    \end{equation}
then
$
 \tilde{s}_{ij}(u)= \sdet_q(S_J(q^Mu))  z_{ij}(u)$.
Using the explicit formula for Sklyanin determinants,  we have that
\begin{equation}
 \sdet_q (\tilde S(u))= \Pi_{i=0}^{N-1} \sdet_q(S_{J}(q^{M-2i}u)) \sdet_q (Z(u)).
    \end{equation}

It follows from Jacobi's identity that
\begin{equation}
  \sdet_{q^{-1}}(X_{11}(u)) \sdet_{q}(Z(q^{N-2}u))=1.
    \end{equation}
This implies that
\begin{equation}
 \sdet_{q}(Z(u))=\sdet_q(S_{J}(q^M)u)^{-1} \sdet_q(S(q^Mu)).
\end{equation}
Therefore,
\begin{equation}
 \sdet_q (\tilde S(u))= \Pi_{i=1}^{N-1} \sdet_q(S_{J}(q^{M-2i}u))  \sdet_q(S(q^Mu)).
    \end{equation}

\end{proof}

\subsection{Liouville formula}

In this section, we regard the twisted $q$-Yangians as subalgebras of the
quantum affine algebra at level $c=0$.
Recall that $S(u)=L^-(u)G \L^+(u^{-1})^{t}$ ,$\ol S(u)=L^+(u)G  L^-(u^{-1})^{t}$. In the orthogonal case, $G=I$.
In the following we derive a common reflection relation involving both $S(u)$ and $\ol S(u)$ which will be used in the Liouville formula.
  \begin{proposition} We have that
\begin{equation}\label{reflection relation1}
  \begin{split}
R(u,v)\overline S_{1}(u)R^{t}(u^{-1}, v)S_{2}(v)=S_{2}(v)R^{t}(u^{-1}, v)\overline S_{1}(u)R(u, v).
 \end{split}
  \end{equation}
\end{proposition}

\begin{proof} We consider the equation in the quantum affine algebra. The left side can be written as
\begin{equation}
  \begin{aligned}
  &R(u,v)\ol S_{1}(u)R^{t}(u^{-1},v)S_{2}(v)\\
&=R(u, v)L^+_1(u)G_1  L^-_1(u^{-1})^{t}R^{t}(u^{-1},  v)L^-_2(v)G_2 L^+_2(v^{-1})^{t}.
 \end{aligned}
  \end{equation}
 Taking transposition $t_1$ to the $RTT$ equation with respect to $L^-(u)$,
  we get that
   \begin{equation}
L^-_1(u)^t R  (u,  v) ^t L^-_2(v)= L^-_2(v)    R  (u,  v)^tL^-_1(u)^t.
  \end{equation}
  Thus,
  \begin{equation}
  \begin{aligned}
    &R(u,v)\ol S_{1}(u)R^{t}(u^{-1},v)S_{2}(v)\\
    &=R(u, v)L^+_1(u)G_1  L^-_2(v)  R^{t}(u^{-1},  v) L^-_1(u^{-1})^{t} G_2 L^+_2(v^{-1})^{t}\\
    &=R(u, v)L^+_1(u) L^-_2(v) G_1   R^{t}(u^{-1},  v) G_2 L^-_1(u^{-1})^{t}  L^+_2(v^{-1})^{t}.
 \end{aligned}
  \end{equation}
  Using the $RTT$ relation, this is equal to
  \begin{equation}
    \begin{split}
 L^-_2(v) L^+_1(u)  R(u, v) G_1   R^{t}(u^{-1},  v) G_2 L^-_1(u^{-1})^{t}  L^+_2(v^{-1})^{t}.\\
   \end{split}
    \end{equation}
    Since the matrix $G$ satisfies the reflection relation in both cases, we get
    \begin{equation}
      \begin{aligned}
        &R(u,v)\ol S_{1}(u)R^{t}(u^{-1},v)S_{2}(v)\\
        &=
   L^-_2(v) L^+_1(u)  G_2   R^{t}(u^{-1},  v)  G_1  R(u, v) L^-_1(u^{-1})^{t}  L^+_2(v^{-1})^{t}.
     \end{aligned}
      \end{equation}

   Taking transposition $t_1$ and $t_2$ consecutively to the equation
 \begin{equation}
 R  (u,  v) L^+_1(u) L^-_2(v)=  L^-_2(v)  L^+_1(u)  R  (u,  v),
  \end{equation}
we get that
   \begin{equation}
L^+_1(u)^t L^-_2(v)^t R  (u,  v) ^{t_1t_2} =     R  (u,  v) ^{t_1t_2} L^-_2(v)^t  L^+_1(u)^t.
  \end{equation}
   Since $R  (u,  v) ^{t_1t_2}=P R  (u,  v)P$,
     \begin{equation}
L^+_2(u)^t L^-_1(v)^t R  (u,  v)   =    R  (u,  v)  L^-_1(v)^t L^+_2(u)^t.
  \end{equation}

As $R(u,  v)$ is equal to $R  (v^{-1},  u^{-1})$ up to constant,
replacing $u$ and $v$ by
$v^{-1}$ and $u^{-1}$ respectively we have that
\begin{equation}
L^+_2(v^{-1})^t L^-_1(u^{-1})^t R  (u,  v)   =    R  (u,  v)   L^-_1(u^{-1})^t L^+_2(v^{-1})^t.
  \end{equation}
Then
\begin{equation}
  \begin{aligned}
    &R(u,v)\ol S_{1}(u)R^{t}(u^{-1},v)S_{2}(v)\\
    &=L^-_2(v) L^+_1(u)  G_2   R^{t}(u^{-1},  v)  G_1    L^+_2(v^{-1})^{t} L^-_1(u^{-1})^{t}  R(u, v)\\
&=L^-_2(v) G_2 L^+_1(u)     R^{t}(u^{-1},  v)      L^+_2(v^{-1})^{t} G_1 L^-_1(u^{-1})^{t}  R(u, v).
 \end{aligned}
  \end{equation}

The $RTT$ relation with respect to $L^+(u)$ is equivalent to
      \begin{equation}
       P R  (u,  v)P L^+_2(u) L^+_1(v) =  L^+_1(v)  L^+_2(u)  PR  (u,  v)P.
         \end{equation}
Taking transposition $t_2$ we get
\begin{equation}
  L^+_2(u)^t R ^t (u,  v) L^+_1(v) =  L^+_1(v)  R ^t (u,  v) L^+_2(u)^t.
    \end{equation}
Replacing $u$ and $v$ by $v^{-1}$ and $u $ respectively we have
    \begin{equation}
      L^+_2(v^{-1})^t R ^t (v^{-1},  u) L^+_1(u) =  L^+_1(u)  R ^t (v^{-1},  u) L^+_2(v^{-1})^t.
        \end{equation}

Note that
$R ^t (v^{-1},  u)$ is equal to $R ^t (u^{-1},  v)$ up to a constant.
Therefore,
\begin{equation}
  \begin{aligned}
    &R(u,v)\ol S_{1}(u)R^{t}(u^{-1},v)S_{2}(v)\\
&=L^-_2(v) G_2  L^+_2(v^{-1})^{t}  R^{t}(u^{-1},  v)  L^+_1(u)  G_1 L^-_1(u^{-1})^{t}  R(u, v)\\
&=S_{2}(v)R^{t}(u^{-1}, v)\overline S_{1}(u)R(u, v).
 \end{aligned}
  \end{equation}
\end{proof}

Multiplying from both sides of the relation \eqref{reflection relation1} by the inverse of
$R(u, v)$, $\overline S_{1}(u)$ and $R^{t}(u^{-1},v)$,
we get that
\begin{equation}
    \begin{split}
   S_{2}(v) R^{-1}(u,v)\overline S_{1}(u)^{-1} R^{t}(u^{-1}, v)^{-1} \\
   = R^{t}(u^{-1}, v)^{-1} \overline S_{1}(u)^{-1}R^{-1}(u,v)S_{2}(v).
   \end{split}
    \end{equation}

Taking $u=q^{-2N}v^{-1}$ and  using the relations \eqref{inverse of R} and  \eqref{cross rel} , we have
\begin{equation}
  \begin{split}
 D_{1}^{-1} S_{2}(v) \wt R'(q^{-2N}v^{-2})  \overline{S}_{1}(q^{-2N}v^{-1})^{-1} D_1   Q \\
=QD_{1}^{-1}
\overline{S}_{1}(q^{-2N}v^{-1})^{-1}  \wt R '(q^{-2N}v^{-2})  S_{2}(v)D_1  .
 \end{split}
  \end{equation}
  where  $\wt R'(x)$ is the inverse of $R(x)$ up to constant, and it is the $R$ matrix obtained from $\wt R(x)$ by replacing $q$ with $q^{-1}$
 Replacing $v$ with $uq^{-2N}$.
  \begin{equation}\label{Liouville}
    \begin{split}
    D_{1}^{-1}   S_{2}(q^{-2N}u)  \wt R'(q^{2N}u^{-2})   \overline{S}_{1}(u^{-1})^{-1}  D_1   Q\\
  =Q   D_{1}^{-1} \overline{S}_{1}(u^{-1})^{-1} \wt R' (q^{2N}u^{-2})   S_{2}(q^{-2N}u) D_{1} .
   \end{split}
    \end{equation}
  Since $Q$ is a one-dimensional operator satisfying $Q^2=NQ$,
  the equation must be $Q$ times a series in $u^{-1}$ with coefficients in $Y_q^{\tw}(\g_N)$.
  We denote the series by $\alpha_N(u)y(u)$, where
   \begin{equation}
    \alpha_N(u)= \begin{cases}
       1,  & \quad \text{Case }(\mathfrak{o}_N),\\
       \frac{q^2 u^2-1}{q^2-u^2}, &\quad \text{Case }(\mathfrak{sp}_N).
      \end{cases}
    \end{equation}
We now fix $y(u)$ by Sklyanin determinants.
\begin{theorem} On the algebra $Y_q^{\tw}(\g_N)$, we have that
\begin{equation}
  y(u)=\frac{\sdet_q(S(q^{-2}u))}{\sdet_q(S(u))}.
\end{equation}
\end{theorem}

\begin{proof}
We write $S(u)=L^-(u)G \L^+(u^{-1})^{t}$ ,$\ol S(u)=L^+(u)G  L^-(u^{-1})^{t}$, then
\begin{equation}
  \begin{split}
    \alpha_N(u)y(u)Q
=Q D_{1}^{-1} {(L^-_1(u)^t)}^{-1}  G_1^{-1} L^+_1(u^{-1})^{-1} \\
 \wt R'(q^{2N}u^{-2})
L^-_2(q^{-2N}u)G_2 L^+_2(q^{2N}u^{-1})^{t} D_{1}.\\
 \end{split}
  \end{equation}

Multiplying the inverse of $R(u,v)$ and $ L^+_1(u) $ to the relation
\begin{equation}
  R  (u,  v) L^+_1(u) L^-_2(v)=  L^-_2(v)  L^+_1(u)  R  (u,  v)
   \end{equation}
we have that
\begin{equation}
  \begin{split}
L^-_2(v)R(u,v)^{-1}L^+_1(u)^{-1} = L^+_1(u)^{-1}R(u,v)^{-1}L^-_2(v).
 \end{split}
  \end{equation}
Thus
\begin{equation}
  \begin{split}
    \alpha_N(u)  y(u)Q
=Q D_{1}^{-1} {(L^-_1(u)^t)}^{-1}  G_1^{-1} L^-_2(q^{-2N}u)  \\
 \wt R'(q^{2N}u^{-2})  L^+_1(u^{-1})^{-1}
G_2 L^+_2(q^{2N}u^{-1})^{t} D_{1}\\
=Q D_{1}^{-1} {(L^-_1(u)^t)}^{-1}   D_1  L^-_2(q^{-2N}u) D_1^{-1} G_1^{-1}  \\
 \wt R'(q^{2N}u^{-2})  G_2 L^+_1(u^{-1})^{-1} D_{1}
 L^+_2(q^{2N}u^{-1})^{t}. \\
 \end{split}
  \end{equation}

It follows from Proposition \ref{center z relation} that
\begin{equation}
  \begin{split}
Q D_1^{-1} (L^{-}_1(u)^t)^{-1}D_1 L^{-}_2(q^{-2N}u)=Qz^{-}(u) .
 \end{split}
  \end{equation}
Then
\begin{equation}
  \begin{aligned}
   & y(u)Q\\
&=z^-(u)Q D_1^{-1} G_1^{-1}
 \wt R'(q^{2N}u^{-2})  G_2 L^+_1(u^{-1})^{-1}
 L^+_2(q^{2N}u^{-1})^{t} D_{1}.
 \end{aligned}
  \end{equation}

  By direct computation we have that
  \begin{equation}
    \begin{aligned}
      &Q D_1^{-1} G_1^{-1}
     \wt R'(q^{2N}u^{-2})  G_2 D_1 \\
     &=   D_{1}^{-1}  G_2  \wt R '(q^{2N}u^{-2})   G_1^{-1}  D_1   Q=Q\beta_N(u)
   \end{aligned}
    \end{equation}
where the constant $\beta_N(u)$ is given by
\begin{equation}
  \beta_N(u)=\begin{cases}
     1,   &\quad \text{Case }(\mathfrak{o}_N),\\
     \frac{q^2u^2-q^{2N}}{q^{2N+2}-u^2}, & \quad \text{Case }(\mathfrak{sp}_N).
    \end{cases}
  \end{equation}

By definition of $z^{+}(u)$, we have
\begin{align}\notag
&Q  D_1^{-1}
L^+_1(u^{-1})^{-1} D_{1}
L^+_2(q^{2N}u^{-1})^{t} \\
&=Q L^+_1(q^{2N}u^{-1}) D_1^{-1}
L^+_1(u^{-1})^{-1} D_{1}\\ \notag
& =z^+(q^{2N}u^{-1})^{-1}.
  \end{align}
On the other hand,
\begin{equation}
    \begin{aligned}
y(u)&=\frac{\beta_N(u)}{\alpha_N(u)}z^-(u)z^+(q^{2N}u^{-1})^{-1}\\
&=\frac{\beta_N(u)}{\alpha_N(u)} \frac{{\det}_{q}(L^-(q^{-2}u))  {\det}_{q}L^+(q^{2N}u^{-1})} {{\det}_{q}(L^-(u)) {\det}_{q}L^+(q^{2N-2}u^{-1})}.
   \end{aligned}
    \end{equation}
Now it is easy to see that
\begin{equation}
 \frac{\beta_N(u)}{\alpha_N(u)} = \frac{\gamma_N(q^{-2}u)}{\gamma_N(u)}.
  \end{equation}
Thus,
    \begin{equation}
      \begin{aligned}
  y(u)&=\frac{\beta_N(u)}{\alpha_N(u)}z^-(u)z^+(q^{2N}u^{-1})^{-1}\\
  &=\frac{\sdet_q(S(q^{-2}u))}{\sdet_q(S(u))}.
     \end{aligned}
      \end{equation}
      \end{proof}

\bigskip
\centerline{\bf Acknowledgments}
\medskip
We would like to thank Alexander Molev for stimulating discussions.
The work is supported in part by the National Natural Science Foundation of China (grant nos.
12001218 and 12171303), the Simons Foundation (grant no. MP-TSM-00002518),
and the Fundamental Research Funds for
 the Central Universities (grant nos. CCNU22QN002 and
CCNU22JC001).

\bibliographystyle{amsalpha}

\end{document}